\documentclass[journal]{IEEEtran}
\usepackage{amsmath,amssymb,amsfonts}
\usepackage{amsthm}
\usepackage{graphicx}
\usepackage{textcomp}
\usepackage{color}
\usepackage{xcolor}
\usepackage{multirow}
\usepackage{subfigure}
\usepackage{multicol,lipsum}
\usepackage[utf8x]{inputenc}
\usepackage{epsfig}
\usepackage{algorithm}
\usepackage{algorithmic}
\usepackage{multirow}
\usepackage{amsmath}
\usepackage{hyperref}
\usepackage{mathtools}

\usepackage{bm}
\usepackage{listings}
\newtheorem{remark}{Remark}
\newtheorem{definition}{Definition}
\newtheorem{theorem}{Theorem}
\newtheorem{proposition}{Proposition}
\newtheorem{lemma}{Lemma}
\newtheorem{example}{Example}
\newtheorem{corollary}{Corollary}

\usepackage{balance}
\usepackage{textcomp}
\usepackage{xcolor}
\def\BibTeX{{\rm B\kern-.05em{\sc i\kern-.025em b}\kern-.08em
    T\kern-.1667em\lower.7ex\hbox{E}\kern-.125emX}}
    
\newcommand{\oomit}[1]{}

\begin{document}

\title{Reach-Avoid Analysis for Polynomial Stochastic Differential Equations}
\author{Bai Xue$^{1}$ and Naijun Zhan$^{1,2}$ and Martin Fr\"anzle$^{3}$
\IEEEauthorblockA{ \\ \textit{1. State Key Lab. of Computer Science, Institute of Software, CAS, Beijing, China}\\
\{xuebai,znj\}@ios.ac.cn}\\
\textit{2, University of Chinese Academy of Sciences,  CAS, Beijing, China}\\
\textit{3. Carl von Ossietzky Universit\"at, Oldenburg, Germany\\
        {martin.fraenzle@uol.de}}
}

\IEEEtitleabstractindextext{%
\begin{abstract}
In this paper we propose a novel semi-definite programming approach that solves reach-avoid problems over open (i.e., not bounded a priori) time horizons for dynamical systems modeled by polynomial stochastic differential equations. The reach-avoid problem in this paper is a probabilistic guarantee: we approximate from the inner a $p$-reach-avoid set, i.e., the set of initial states guaranteeing with probability larger than $p$ that the system eventually enters a given target set while remaining inside a specified safe set till the target hit. Our approach begins with the construction of a bounded value function, whose strict $p$ super-level set is equal to the $p$-reach-avoid set. This value function is then reduced to a twice continuously differentiable solution to a system of equations. The system of  equations  facilitates the construction of a semi-definite program using sum-of-squares decomposition for multivariate polynomials and thus the transformation of nonconvex reach-avoid problems into a convex optimization problem. The semi-definite program can be solved efficiently in polynomial time 
with many existing powerful algorithms such as interior point methods and off-the-shelf software packages. We would like to point out that our approach can straightforwardly be specialized to 
address classical safety verification by, a.o., stochastic barrier certificate methods and reach-avoid analysis for ordinary differential equations. In addition, several examples are provided to  demonstrate theoretical and algorithmic developments of the proposed method.
\end{abstract}

}

\maketitle

\IEEEdisplaynontitleabstractindextext

\IEEEpeerreviewmaketitle

\section{Introduction}
\label{sec:intro}
Reach-avoid analysis combines the construction of safety and specific progress guarantees for dynamical systems, as it addresses guarantees for both the eventual reach of desirable states and avoidance of unsafe states. It is employed in diverse engineering applications including collision avoidance \cite{malone2014stochastic} and motion planning \cite{lesser2013stochastic}. Algorithmic methods for computing reach-avoid guarantees have consequently been widely studied, e.g.\ in \cite{margellos2011,korda2014controller,xue2016reach}. In the qualitative setting, reach-avoid analysis generally attempts to determine a set of initial states driving the system to a desirable target set with certainty over either finite time horizons (i.e., given a fixed upper bound on the first hitting time) or open time horizons (i.e., unknown upper bound on the first hitting time) while reliably avoiding a set of unsafe states before hitting the target set.

In existing literature, difference equations (DEs) and ordinary differential equations (ODEs) are frequently used to model deterministic systems, and a large body of work has been pursued addressing reach-avoid analysis for such systems, like \cite{margellos2011,fisac2015reach,han2018semidefinite,xue2019,xue2020}. But DEs or ODEs constitute confined models of real-world systems, as stochastic processes are central to many phenomena in physics, engineering, biology and other disciplines \cite{bartlett1978,pola2003stochastic,franzle2011measurability}. When considering stochastic systems, i.e., dynamical systems involving stochastic processes, solving the reach-avoid problem qualitatively in a non-stochastic manner usually gives pessimistic answers, since in general resultant bounds on the values of stochastic inputs will be overly conservative. It indeed is  natural to formulate and solve probabilistic variants of reach-avoid problems. The notion of $p$-reach-avoid reachability used herein reflects this probabilistic perspective. It requires identifying the set of initial states that guarantee with probability being larger than $p$ that the dynamical system reaches a given target set. It was studied in various engineering applications such as a spacecraft rendezvous and docking problem \cite{lesser2013stochastic} and a Zermelo navigation problem \cite{esfahani2016}.
In the literature, 
Markov chains, Markov decision processes and stochastic differential equations (SDEs) are among the most commonly 
used models for  stochastic processes. The $p$-reach avoid problem of 
the first two   over both finite time horizons and open time horizons has 
been studied in, e.g., \cite{abate2008,gleason2017,vinod2021stochastic,xue2021reach}. The quest for generalizations to continuous-time dynamical system models, especially SDEs, remains largely unanswered.

In this paper we therefore investigate the $p$-reach-avoid problem for systems modeled by polynomial SDEs and the focus is on the computation of inner-approximations of the exact $p$-reach-avoid set over open time horizons. The $p$-reach-avoid set is the set of initial states guaranteeing with probability being larger than $p$ an eventual hit of a desirable target set while staying inside a designated safe set prior to hitting the target. 
The inner-approximation problem is reduced to a semi-definite programming problem in our approach. 
The construction of the semi-definite program originates from a value function whose strict $p$ super-level set equals the $p$-reach-avoid set. The particular value function is defined based on an appropriately stopped variant of the dynamical process under investigation, and is shown to be the unique twice continuously differentiable solution to an effectively constructed system of equations. Based on the obtained system of equations, we further construct a system of inequalities and encode them into semi-definite constraints using the sum of-squares decomposition for multivariate polynomials. This system of constraints finally 
results in a semi-definite program whose solution under-approximates the exact $p$-reach-avoid set. The performance of the proposed approach is illustrated by several examples. 

The main contributions of this work are summarized below.
\begin{enumerate}
    \item An innovative system of equations is proposed for characterizing the exact $p$-reach-avoid set over open time horizons for systems modelled by SDEs. The system of equations plays a fundamental role in our methodology, since it explains the origins of the constructed convex program for inner-approximating the $p$-reach-avoid set. Besides, the proposed system of equations can also be used to construct a set of constraints for addressing the classical safety verification problem of SDEs that can be solved with stochastic barrier certificate methods in \cite{PrajnaJP07}. However, these two methods are disparate, as commented in Remark \ref{UN}. In this article, we  will not benchmark the performance of our method on the classical safety verification problem, since the focus of this work is on inner-approximating \emph{reach-avoid} sets. However, we theoretically compare the set of constraints constructed by our method and the one in \cite{PrajnaJP07}. It concludes that  the set of constraints constructed by our method is more expressive than the one in \cite{PrajnaJP07}.
    \item A novel convex programming based approach is proposed for inner-approximating the $p$-reach-avoid set over open time horizons, which solves the complicated non-convex reachability problem arising in dynamical systems and control theory by solving a single semi-definite program. The semi-definite program is relatively simple and can be efficiently solved by many existing  powerful algorithms and off-the-shelf software packages.
    \item  Our semi-definite program is implemented based on the sum-of-squares module of YALMIP \cite{lofberg2004} and the semi-definite programming solver Mosek \cite{mosek2015mosek}. Several examples are presented to demonstrate theoretical and algorithmic developments of our method.
\end{enumerate}

\subsection*{\textbf{Related Work}}
\label{rw}
In the discrete-time setting, the reach-avoid problem has been widely studied for Markov decision processes, e.g., \cite{abate2007,abate2008,summers2010,vinod2021stochastic}. The corresponding reach-avoid problem is  normally  reformulated as a dynamic programming problem  by introducing indicator functions for the sets of target and unsafe states. A straightforward way to numerically approximate the value function of the reach-avoid dynamic program is by constructing a grid state space \cite{abate2007}. Under suitable assumptions on the Markov decision process limiting local variation of its kernel, gridding approaches can provide rigorous performance guarantees on the resulting approximate solution, yet they suffer from the ``curse of diemensionality'' of an exponential complexity   in the dimension of the state space. In order to alleviate the complexity of  gridding-based approximations, a semi-definite programming based method was proposed in \cite{drzajic2017,xue2021reach}.

In the continuous-time setting, reach-avoid analysis for systems modeled by SDEs has attracted increasing attention since the first study of reach-avoid verification for SDEs over open time horizons in~\cite{PrajnaJP07}. In \cite{PrajnaJP07} a typical supermartingale was employed as a stochastic barrier certificate followed by
computational conditions derived from Doob's martingale
inequality~\cite{karatzas2014brownian}. It provides \textit{an upper bound on the probability} of reaching a set of unsafe states for a stochastic system starting from a set of legal initial state. Recently, via removing the requirement of reaching target sets, the stochastic barrier certificate-based method was extended to cater for bounding the probability of leaving a finite region of state space over a given finite time by leveraging a relaxed formulation termed $c$-martingale for locally stable systems in \cite{steinhardt2012finite}, and further extended to the problem of controller synthesis for ensuring that unsafe probability is below a threshold in \cite{santoyo2021barrier}. The differences between the present work and \cite{PrajnaJP07} are twofold. One is that the set of permissible initial states is synthesized from the system dynamics and the desired reach-avoid property rather than analysing a given state set as in \cite{PrajnaJP07}. The other one is that \textit{a lower bound on the probability} of reaching a specified, considered to be desirable, set of target states is computed in the present work. The method in the present work can nevertheless straightforwardly be modified to also cover extended computation of an upper bound on the probability as in \cite{PrajnaJP07}, as pointed out in Remark \ref{UN}.

Another known method for studying the reach-avoid problem is the Hamilton-Jacobi reachability one. Hamilton-Jacobi reachability method addresses reach-avoid problems by exploiting the link to optimal control through viscosity solutions of Hamilton-Jacobi type equations. It extends the use of Hamilton-Jacobi equations, which are widely used in optimal control theory, to perform reachability analysis over both finite time horizons \cite{esfahani2016} and open time horizons \cite{Koutsoukos2008}. However, grid-based numeric approaches, e.g., the finite difference method in~\cite{Koutsoukos2008} and the level set method in~\cite{MT05}, are traditionally used to solve these equations, rendering the Hamilton-Jacobi reachability method computationally infeasible for even moderate sized systems. Furthermore, such methods cannot guarantee that the computed result is an outer- or inner-approximation of the reach-avoid set. In \cite{bujorianu2007new} the reach-avoid problem was reduced to a problem of solving semi-continuous solutions to some variational inequalities. The lack of continuous solutions is an obstacle to actually solving these inequalities solving. In contrast, we propose an innovative system of partial differential equations, which is different from existing Hamilton-Jacobi equations in literature, to characterize reach-avoid sets for SDEs. The proposed equations facilitate the construction of semi-definite programs, which can be efficiently solved by interior-point methods in polynomial time, for computing guaranteed inner-approximations of reach-avoid sets.

Recently, a moment-based method, which is also a convex programming based method, was proposed for studying reach-avoid problems over finite time horizons for SDEs in \cite{sloth2015safety}, and a semi-definite programming method derived from Feynman-Kac formula was proposed for analysing avoid problems (without the requirement of reaching target sets) over finite time horizons in~\cite{liureachability} that is algebraically over- and under-approximating the staying probability in a given safety area. Different from the above two methods, 
our method in this paper addresses the reach-avoid problem over open time horizons rather than finite time horizons.

The structure of this paper is as follows: Section \ref{sec:preliminaries} introduces stochastic systems and reach-avoid problems of interest. After detailing the derivation of the system of equations for characterizing the $p$-reach-avoid set in Subsection \ref{equations}, we introduce our semi-definite programming method for inner-approximating the $p$-reach-avoid set in Subsection \ref{SDPI}. In Section \ref{sec:exam} we demonstrate the performance of our approach on several examples and finally provide conclusions as well as future work in Section \ref{sec:conclusion}.
\section{Preliminaries}
\label{sec:preliminaries}
We start our exposition by formally presenting polynomial SDEs and $p$-reach-avoid sets of interest. Beforehand we introduce basic notions used throughout this paper: $\mathbb{R}_{\geq 0}$ stands for the set of nonnegative reals and $\mathbb{R}$ for the set of real numbers. For a set $\Delta$, $\Delta^c$, $\overline{\Delta}$ and $\partial \Delta$ denote the complement, the closure and the boundary of the set $\Delta$, respectively. $\bigwedge$ and $\bigvee$ denote the logical operation of conjunction and disjunction, respectively. $\mathbb{R}[\cdot]$ denotes the ring of polynomials in variables given by the argument. Vectors are denoted by boldface letters. $\sum[\bm{x}]$ is used to represent the set of sum-of-squares polynomials over variables $\bm{x}$, i.e.,
\[\sum[\bm{x}]=\{p\in \mathbb{R}[\bm{x}]\mid p=\sum_{i=1}^{k'} q_i^2,q_i\in \mathbb{R}[\bm{x}],i=1,\ldots,k'\}.\]

Let $(\Omega,\mathcal{F},\mathbb{P})$ be a probability space \cite{oksendal2003}, where $\Omega$ is the sample space, $\mathcal{F}\subseteq 2^{\Omega}$ is a $\sigma$-algebra on $\Omega$, and $\mathbb{P}: \mathcal{F}\rightarrow [0,1]$ is a probability measure on the measurable space $(\Omega,\mathcal{F})$. A random variable $\bm{X}$ defined on the probability space $(\Omega,\mathcal{F},\mathbb{P})$ is an $\mathcal{F}-$measurable function $\bm{X}: \Omega\rightarrow \mathbb{R}^n$; its expectation (w.r.t. $\mathbb{P}$) is denoted by $E[\bm{X}]$. Every random variable $\bm{X}$ induces a probability measure $\mu_{\bm{X}}: \mathcal{B}\rightarrow [0,1]$ on $\mathbb{R}^n$, defined as $\mu_{\bm{X}}(B)=\mathbb{P}(\bm{X}^{-1}(B))$ for a Borel set $B$ in the Borel $\sigma-$algebra $\mathcal{B}$ on $\Omega$. $\mu_{\bm{X}}$ is called the distribution of $\bm{X}$, and its support set is $\mathtt{supp}{\mu_{\bm{X}}}=\overline{\{B\in \mathcal{B}\mid \mu_{\bm{X}}(B)>0\}}$. The support set for a real-valued function $f(\cdot): \Delta \rightarrow \mathbb{R}$ is the closure of the subset of $\Delta$, where $f$ is non-zero, i.e.,  $\mathtt{supp}{f}=\overline{\{\bm{x}\in \Delta \mid f(\bm{x})\neq 0}\}$. 
A continuous-time stochastic process is a parameterized collection of random variables $\{\bm{X}(t,\bm{w}),t\in T\}$ where the parameter space $T$ can be either the halfline $\mathbb{R}_{\geq 0}$ or an interval $[a,b]$. Note that for each $t\in T$ fixed we have a random variable $\bm{X}(t,\cdot): \Omega\rightarrow \mathbb{R}^n$. On the other hand, fixing $\bm{w}\in \Omega$ we can consider the function $\bm{X}(\cdot,\bm{w}): T\rightarrow \mathbb{R}^n$, which is called a path of the stochastic process. A collection $\{\mathcal{F}_t\mid t\geq 0\}$ of $\sigma-$algebra of sets in  $\mathcal{F}$ is a filtration if $\mathcal{F}_t\subseteq \mathcal{F}_{t+s}$ for $t,s\in \mathbb{R}_{\geq 0}$ (Intuitively, $\mathcal{F}_t$ carries the information known to an observer at time $t$.). A random variable $\tau:\Omega\rightarrow \mathbb{R}_{\geq 0}$ is called a stopping time w.r.t. some filtration $\{\mathcal{F}_t\mid t\geq 0\}$ of $\mathcal{F}$ if $\{\tau\leq t\}\in \mathcal{F}_t$ for all $t \geq 0$. Note that a constant time is always a stopping time. 

We consider stochastic systems modeled by time-homogeneous SDEs of the form
\begin{equation}
\label{sde}
d \bm{X}(t,\bm{w})=\bm{b}(\bm{X}(t,\bm{w}))dt +\bm{\sigma}(\bm{X}(t,\bm{w})) d \bm{W}(t,\bm{w}), t\geq 0,
\end{equation}
where $\bm{X}(\cdot,\cdot): T \times \Omega \rightarrow \mathbb{R}^n$  is an $n$-dimensional continuous-time stochastic process, $\bm{W}(\cdot,\cdot): T\times \Omega \rightarrow \mathbb{R}^m$ is an $m$-dimensional Wiener process (standard Brownian motion), the mapping $\bm{b}(\cdot): \mathbb{R}^n \rightarrow \mathbb{R}^n$ is a vector-valued polynomial (i.e., each of its components is a polynomial), and $\bm{\sigma(\cdot)}: \mathbb{R}^n\rightarrow \mathbb{R}^{n\times m}$ is a matrix-valued polynomial, i.e., each of its components is a polynomial.

Since each component of both $\bm{b}(\bm{x})$ and $\bm{\sigma}(\bm{x})$ is polynomial over $\bm{x}$, satisfying locally Lipschitz conditions, then given an initial state $\bm{x}_0\in \mathbb{R}^n$, an SDE of the form \eqref{sde} has a unique (maximal local) strong solution over some time interval $[0,T^{\bm{x}_0}(\bm{w}))$ for $\bm{w}\in \Omega$ [Lemma 2.2, \cite{song2015noise}], where $T^{\bm{x}_0}(\bm{w})$ is a positive real value. We denote it as $\bm{X}^{\bm{x}_0}(\cdot,\bm{w}): [0,T^{\bm{x}_0}(\bm{w}))\times \Omega \rightarrow \mathbb{R}^n$, which satisfies the stochastic integral equation 
\begin{equation*}
\label{integral}
\begin{split}
\bm{X}^{\bm{x}_0}(t,\bm{w})=\bm{x}_0&+\int_{0}^t \bm{b}(\bm{X}^{\bm{x}_0}(s,\bm{w})) ds \\
&+\int_{0}^t \bm{\sigma}(\bm{X}^{\bm{x}_0}(s,\bm{w})) d \bm{W}(s,\bm{w})
\end{split}
\end{equation*}
for $t\in [0,T^{\bm{x}_0}(\bm{w}))$.

The infinitesimal generator underlying system \eqref{sde} is presented in Definition \ref{infi}.
\begin{definition}\cite{oksendal2003}
\label{infi}
Let $\bm{X}^{\bm{x}}(t,\bm{w})$ be a time-homogeneous It\^{o} diffusion given by {\rm SDE} \eqref{sde} with initial state $\bm{x} \in \mathbb{R}^n$. The infinitesimal generator $\mathcal{A}$ of $\bm{X}^{\bm{x}}(t,\bm{w})$ is defined by 
\begin{equation}
\begin{split}
&\mathcal{A}f(\bm{x})=\lim_{t\rightarrow 0}\frac{E[f(\bm{X}^{\bm{x}}(t,\bm{w}))]-f(\bm{x})}{t}\\
&=\sum_{i} b_i(\bm{x}) \frac{\partial f(\bm{x})}{\partial x_i}+\frac{1}{2}\sum_{i,j}(\bm{\sigma}\bm{\sigma}^{\top})_{ij}(\bm{x})\frac{\partial^2 f(\bm{x})}{\partial x_i \partial x_j}.
\end{split}
\end{equation}
for any $f\in \mathcal{C}^{2}(\mathbb{R}^n)$, where $\mathcal{C}^{2}(\mathbb{R}^n)$ denotes the set of twice continuously differentiable functions.
\end{definition}

As a stochastic generalization of the Newton-Leibniz axiom, Dynkin's formula gives the expected value of any suitably smooth function of an It\^o diffusion at a stopping time.
\begin{theorem}[Dynkin's formula, \cite{oksendal2003}]
\label{dyn_theom}
Let $\bm{X}^{\bm{x}}(t,\bm{w})$ be a time-homogeneous It\^o diffusion given by {\rm SDE} \eqref{sde} with the initial state $\bm{x}\in \mathbb{R}^n$. Suppose $\tau$ is a stopping time with $E[\tau]<\infty$, and $f\in \mathcal{C}^{2}(\mathbb{R}^n)$ with compact support. Then  
\begin{equation}
\label{dyn}
E[f(\bm{X}^{\bm{x}}(\tau,\bm{w}))]=f(\bm{x})+E[\int_{0}^{\tau} \mathcal{A}f(\bm{X}^{\bm{x}}(s,\bm{w}))ds].
\end{equation}
\end{theorem}

In Theorem \ref{dyn_theom}, if we consider a twice continuously differentiable function $f$ defined on a bounded set  $B\subseteq \mathbb{R}^n$, i.e., $f(\bm{x}) \in \mathcal{C}^2(B)$, $f$ can be any twice continuously differentiable function $f \in  \mathcal{C}^2(B)$ without the assumption of compact support. In this case, the support of $f$ is of course, compact, since the support of $f$ is always closed and bounded. 
 
 Now, we define the $p$-reach-avoid set, which is a set of initial states such that the stochastic system \eqref{sde} starting from it will touch a compact target set $\mathcal{T}$ in finite time while staying within a bounded and open safe set $\mathcal{X}$ preceding the target first hitting time with probability being larger than $p\in [0,1)$, where 
\begin{equation}
\begin{split}
&\mathcal{T}=\{\bm{x}\in \mathbb{R}^n\mid g(\bm{x})\leq 1\} \text{~and}\\
&\mathcal{X}=\{\bm{x}\in \mathbb{R}^n \mid h_0(\bm{x})< 0 \}
\end{split}
\end{equation}
with $g(\bm{x}), h_0(\bm{x})\in \mathbb{R}[\bm{x}]$ and $\mathcal{T}\subseteq \mathcal{X}$.
\begin{definition}[$p$-Reach-Avoid Set]\label{RNS}
The $p$-reach-avoid set ${\rm RA}_p$ is the set of initial states such that every trajectory of the stochastic system \eqref{sde} originating in it will enter the target set $\mathcal{T}$ at some time $t\in \mathbb{R}_{\geq 0}$ while staying inside the safe set $\mathcal{X}$ over the time horizon $[0,t]$ with probability being larger than $p\in [0,1)$, i.e.,
\begin{equation*}
\label{ras}
{\rm RA}_p=\left\{\bm{x} \in \mathcal{X}\middle|\;
\begin{aligned}
&\mathbb{P}\Big(\exists t \in  \mathbb{R}_{\geq 0}.\big[ \bm{X}^{\bm{x}}(t,\bm{w})\in \mathcal{T} \bigwedge \\
& \forall \tau\in [0,t]. \bm{X}^{\bm{x}}(\tau,\bm{w})\in \mathcal{X}\big]\Big)>p
\end{aligned}
\right\}.
\end{equation*}
An inner-approximation is a subset of the set ${\rm RA}_p$.
\end{definition}

\section{Inner-approximating $p$-Reach-avoid Sets}
\label{sec:inner}
In this section we present our semi-definite programming based approach for inner-approximating the $p$-reach-avoid set ${\rm RA}_p$. The semi-definite program is constructed via relaxing a system of equations, whose twice continuously differentiable solution is equal to a bounded value function with its strict $p$ super-level set being equal to the $p$-reach-avoid set ${\rm RA}_p$. 

\subsection{Characterization of $p$-Reach-Avoid Sets}
\label{equations}
In this subsection we introduce a system of elliptic partial differential equations for characterizing the $p$-reach-avoid set.

The derivation of such equations begins with a value function, which is defined by a new stochastic process $\{\widehat{\bm{X}}^{\bm{x}_0}(t,\bm{w}), t\in \mathbb{R}_{\geq 0}\}$ for $\bm{x}_0\in \overline{\mathcal{X}}$, which is a stopped process corresponding to $\{\bm{X}^{\bm{x}_0}(t,\bm{w}), t\in [0,T^{\bm{x}_0}(\bm{w}))\}$ and the set $\mathcal{X}\setminus \mathcal{T}$, i.e., 
\begin{equation}
\widehat{\bm{X}}^{\bm{x}_0}(t,\bm{w})=
\begin{cases}
&\bm{X}^{\bm{x}_0}(t,\bm{w}), \text{\rm~if~}t<\tau^{\bm{x}_0}(\bm{w}),\\
&\bm{X}^{\bm{x}_0}(\tau^{\bm{x}_0}(\bm{w}),\bm{w}), \text{\rm~if~}t\geq \tau^{\bm{x}_0}(\bm{w}),
\end{cases}
\end{equation}
where \[\tau^{\bm{x}_0}(\bm{w})=\inf\{t\mid \bm{X}^{\bm{x}_0}(t,\bm{w}) \in \partial \mathcal{X} \bigvee \bm{X}^{\bm{x}_0}(t,\bm{w}) \in \mathcal{T}\}\] is the first time of exit of $\bm{X}^{\bm{x}_0}(t,\bm{w})$ from the open set $\mathcal{X}\setminus \mathcal{T}$. It is worth remarking here that if the path $\bm{X}^{\bm{x}_0}(t,\bm{w})$ escapes to infinity in finite time, it must touch the boundary of the bounded safe set $\mathcal{X}$ and thus $\tau^{\bm{x}_0}(\bm{w})\leq T^{\bm{x}_0}(\bm{w})$. The stopped process $\widehat{\bm{X}}^{\bm{x}_0}(t,\bm{w})$ inherits the right continuity and strong Markovian property of $\bm{X}^{\bm{x}_0}(t,\bm{w})$. Moreover, the infinitesimal generator corresponding to $\widehat{\bm{X}}^{\bm{x}_0}(t,\bm{w})$ is identical to the one corresponding to $\bm{X}^{\bm{x}_0}(t,\bm{w})$ on the set $ \mathcal{X}\setminus \mathcal{T}$, and is equal to zero outside of the set $\mathcal{X}\setminus \mathcal{T}$ \cite{kushner1967}. That is, for $v(\bm{x})\in \mathcal{C}^2(\mathbb{R}^n)$, 
\[\mathcal{A}v(\bm{x})=\sum_{i} b_i(\bm{x}) \frac{\partial v(\bm{x})}{\partial x_i}+\frac{1}{2}\sum_{i,j}(\bm{\sigma}\bm{\sigma}^{\top})_{ij}(\bm{x})\frac{\partial^2 v(\bm{x})}{\partial x_i \partial x_j}\] for $\bm{x}\in \mathcal{X}\setminus \mathcal{T}$ and \[\mathcal{A}v(\bm{x})=0\] 
for $\bm{x}\in \partial \mathcal{X}\cup \mathcal{T}$. This will be implicitly assumed throughout this paper.

We observe that the set $\overline{\mathcal{X}}$ is an invariant set for the stochastic process $\widehat{\mathcal{X}}_t^{\bm{x}_0}(\bm{w})$ with $\bm{x}_0\in \overline{\mathcal{X}}$.
\begin{proposition}
\label{three}
If $\bm{x}_0\in \overline{\mathcal{X}}$ and $\bm{w}\in \Omega$, then  \[\widehat{\bm{X}}^{\bm{x}_0}(t,\bm{w})\in \overline{\mathcal{X}}\] for $t\in \mathbb{R}_{\geq 0}$.
\end{proposition}
\begin{proof}
Clearly, if $\bm{x}_0\in \mathcal{T} \cup \partial \mathcal{X}$,  \[\widehat{\bm{X}}^{\bm{x}_0}(t,\bm{w})\in \mathcal{T} \cup \partial \mathcal{X}, \forall t\in \mathbb{R}_{\geq 0}\] holds. 

If $\bm{x}_0\in \mathcal{X}\setminus \mathcal{T}$, one of the following three cases hold:
\begin{enumerate}
\item there exists $t\in \mathbb{R}_{\geq 0}$ such that 
\begin{equation*}
    \begin{split}
\big[\forall \tau\in [0,t). &\widehat{\bm{X}}^{\bm{x}_0}(\tau,\bm{w}) \in \mathcal{X}\setminus \mathcal{T}] \bigwedge\\
&[\forall \tau \in [t, \infty). \widehat{\bm{X}}^{\bm{x}_0}(t,\bm{w}) \in \mathcal{T}\big];
\end{split}
\end{equation*}
\item there exists $t\in \mathbb{R}_{\geq 0}$ such that 
\begin{equation*}
    \begin{split}
\big[\forall \tau\in [0,t). &\widehat{\bm{X}}^{\bm{x}_0}(\tau,\bm{w}) \in \mathcal{X}\setminus \mathcal{T}]  \bigwedge \\
&[\forall \tau \in [t, \infty). \widehat{\bm{X}}^{\bm{x}_0}(t,\bm{w}) \in \partial \mathcal{X}\big];
\end{split}
\end{equation*}
\item $\widehat{\bm{X}}^{\bm{x}_0}(t,\bm{w}) \in \mathcal{X}\setminus \mathcal{T}$ for $t\in \mathbb{R}_{\geq 0}$ holds.
\end{enumerate}
Therefore, the conclusion holds.
\end{proof}

From the proof of Proposition \ref{three}, we conclude that all sample paths of the stochastic process $\{\widehat{\bm{X}}^{\bm{x}_0}(t,\bm{w}),t\in \mathbb{R}_{\geq 0}\}$ for $\bm{x}_0\in \overline{\mathcal{X}}$ can be divided into the following three disjoint groups: 
\begin{enumerate}
\item paths entering $\mathcal{T}$ in finite time. 
\item paths entering $\partial \mathcal{X}$ in finite time;
\item paths staying inside $\mathcal{X}\setminus \mathcal{T}$ for all time.
\end{enumerate}

 Given $\bm{x}\in \overline{\mathcal{X}}$, let $\widehat{\tau}_{\mathcal{T}}^{\bm{x}}(\bm{w})$ be the first hitting time of the target set $\mathcal{T}$ for the path $\widehat{\bm{X}}^{\bm{x}}(t,\bm{w}):\mathbb{R}_{\geq 0}\rightarrow \mathbb{R}^n$, i.e., 
\[\widehat{\tau}_{\mathcal{T}}^{\bm{x}}(\bm{w})=\inf \{t\in \mathbb{R}_{\geq 0}\mid \widehat{\bm{X}}^{\bm{x}}(t,\bm{w}) \in \mathcal{T}\}.\] 
Below we show that the $p$-reach-avoid set ${\rm RA}_p$ is equal to the set of initial states such that the first hitting time of the target set $\mathcal{T}$ for the stochastic process $\{\widehat{\bm{X}}^{\bm{x}_0}(t,\bm{w}),t\in \mathbb{R}_{\geq 0}\}$ is less than infinity with probability being larger than $p$. 
\begin{lemma}
\label{hitting1}
${\rm RA}_p=\{\bm{x}\in \overline{\mathcal{X}}\mid \mathbb{P}(\widehat{\tau}_{\mathcal{T}}^{\bm{x}}(\bm{w})<\infty)>p\}$, where ${\rm RA}_p$ is the $p$-reach-avoid set in {\rm Definition} \ref{RNS}.
\end{lemma}
\begin{proof}
According to Proposition 5 in \cite{bujorianu2007new}, we have $\mathbb{P}(\widehat{\tau}_{\mathcal{T}}^{\bm{x}}(\bm{w})<\infty)=\mathbb{P}(\exists t\in \mathbb{R}_{\geq 0}.\widehat{\bm{X}}^{\bm{x}}(t,w)\in \mathcal{T})$. According to the relationship between stochastic processes $\{\widehat{\bm{X}}^{\bm{x}}(t,\bm{w}), t\in [0,\infty)\}$ and $\{\bm{X}^{\bm{x}}(t,\bm{w}),t\in [0,T^{\bm{x}}(\bm{w}))\}$, we have that $\mathbb{P}(\exists t\in \mathbb{R}_{\geq 0}.\widehat{\bm{X}}^{\bm{x}}(t,w)\in \mathcal{T})=\mathbb{P}\Big(\exists t \in  \mathbb{R}_{\geq 0}.\big[ \bm{X}^{\bm{x}}(t,\bm{w})\in \mathcal{T} \bigwedge \forall \tau\in [0,t]. \bm{X}^{\bm{x}}(\tau,\bm{w})\in \mathcal{X}\big]\Big)$. Therefore, ${\rm RA}_p=\{\bm{x}\in \overline{\mathcal{X}}\mid \mathbb{P}(\widehat{\tau}_{\mathcal{T}}^{\bm{x}}(\bm{w})<\infty)> p\}.$ 
\end{proof}

Now, we present the bounded value function $V(\bm{x}):\overline{\mathcal{X}}\rightarrow \mathbb{R}$, which can be regarded as an ergodic occupation measure or `long-run average’ (e.g., \cite{bujorianu2012stochastic}) and whose strict $p$ super-level set, i.e., $\{\bm{x}\in \overline{\mathcal{X}}\mid V(\bm{x})> p\}$, is equal to the $p$-reach-avoid set ${\rm RA}_p$, as shown in Lemma \ref{value_reach}.
\begin{equation}
\label{value}
V(\bm{x}):=\lim_{t \rightarrow \infty}\frac{\mu([0,t] \times \mathcal{T}\mid \bm{x})}{t},
\end{equation}
where $\mu([0,t] \times \mathcal{T}\mid \bm{x})=E[\int_{0}^t 1_{\mathcal{T}}(\widehat{\bm{X}}^{\bm{x}}(\tau,\bm{w}))d\tau]$ is an occupation measure \cite{bujorianu2012stochastic}, $1_{\mathcal{X}}(\cdot): \mathcal{T}\rightarrow \{0,1\}$ represents the indicator function of the set $\mathcal{T}$, i.e., \[1_{\mathcal{T}}(\bm{x}):=\begin{cases}
   1, \quad \text{if }\bm{x}\in \mathcal{T},\\
   0, \quad \text{if }\bm{x}\notin \mathcal{T}.
\end{cases}\]
Based on occupation measures, \cite{henrion2021moment} investigated the exit time problem of polynomial SDEs using the so-called Lasserre or moment sum of squares hierarchy \cite{lasserre2009moments}. Since $0 \leq 1_{\mathcal{T}}(\bm{x})\leq 1$ over $\mathbb{R}^n$, $0\leq V(\bm{x})\leq 1 \text{~for~} \bm{x}\in \overline{\mathcal{X}}$ and thus $V(\bm{x})$ is bounded over $\overline{\mathcal{X}}$. It is worth remarking here that $\lim_{t \rightarrow \infty}\frac{\mu([0,t] \times \mathcal{T}\mid \bm{x})}{t}$ exists, since $\lim_{t \rightarrow \infty}\frac{\mu([0,t] \times \mathcal{T}\mid \bm{x})}{t}=\sup_{t\in \mathbb{R}_{\geq 0}} E[1_{\mathcal{T}}(\widehat{X}^{\bm{x}}(\tau,\bm{w}))]$, which can also be justified by Proposition 6 in \cite{bujorianu2007new}.

\begin{lemma}
\label{value_reach}
${\rm RA}_p=\{\bm{x}\in \overline{\mathcal{X}} \mid V(\bm{x})> p\}$, where $V(\cdot): \overline{\mathcal{X}}\rightarrow [0,1]$ is the value function in \eqref{value}.
\end{lemma}
\begin{proof}
According to Lemma \ref{hitting1}, we just need to prove that 
\[V(\bm{x})=\mathbb{P}(\widehat{\tau}_{\mathcal{T}}^{\bm{x}}(\bm{w})<\infty).\]

For $t\in \mathbb{R}_{\geq 0}$, according to Fubini's theorem \cite{rosenthal2006first}, we have 
 \[\frac{E[\int_{0}^t 1_{\mathcal{T}}(\widehat{\bm{X}}^{\bm{x}}(\tau,\bm{w}))d\tau]}{t}=\frac{\int_{0}^t \mathbb{P}(\widehat{\bm{X}}^{\bm{x}}(\tau,\bm{w})\in \mathcal{T}) d\tau }{t}.\] Therefore, \[V(\bm{x})=\lim_{t\rightarrow \infty} \frac{\int_{0}^t \mathbb{P}(\widehat{\bm{X}}^{\bm{x}}(\tau,\bm{w})\in \mathcal{T}) d\tau }{t}.\] According to Lemma \ref{eq} shown below, we have \[\lim_{t\rightarrow \infty}\mathbb{P}(\widehat{\bm{X}}^{\bm{x}}(t,\bm{w})\in \mathcal{T})=\mathbb{P}(\widehat{\tau}_{\mathcal{T}}^{\bm{x}}(\bm{w})<\infty).\] 
As a consequence, \[V(\bm{x})=\mathbb{P}(\widehat{\tau}_{\mathcal{T}}^{\bm{x}}(\bm{w})<\infty).\]

Therefore, \[{\rm RA}_p=\{\bm{x}\in \overline{\mathcal{X}} \mid V(\bm{x})> p\}\] according to Lemma \ref{hitting1}.
\end{proof}

\begin{lemma}
\label{eq}
If $\bm{x}_0\in \overline{\mathcal{X}}$, then 
\[\lim_{t\rightarrow \infty}\mathbb{P}(\widehat{\bm{X}}^{\bm{x}_0}(t,\bm{w})\in \mathcal{T})=\mathbb{P}(\widehat{\tau}_{\mathcal{T}}^{\bm{x}_0}(\bm{w})<\infty).\]
\end{lemma}
\begin{proof}
According to Corollary 1 in \cite{tkachev2011infinite} stating that $\mathbb{P}(\exists t\in \mathbb{R}_{\geq 0}.  \widehat{X}^{\bm{x}_0}(t,\bm{w})\in \mathcal{T})=\lim_{t\rightarrow \infty}P(\widehat{X}^{\bm{x}_0}(t,\bm{w})\in \mathcal{T})$, and Proposition 5 in \cite{bujorianu2007new} stating that $\mathbb{P}(\widehat{\tau}_{\mathcal{T}}^{\bm{x}_0}(\bm{w})<\infty)=\mathbb{P}(\exists t\in \mathbb{R}_{\geq 0}.  \widehat{X}^{\bm{x}_0}(t,\bm{w})\in \mathcal{T})$, we have the conclusion.
\end{proof}

From Lemma \ref{value_reach} we conclude that the exact $p$-reach-avoid set can be obtained if the value function $V(\bm{x})$ in \eqref{value} is computed. However, it is challenging, even impossible to compute it directly since it involves the knowledge of analytical solutions to SDE \eqref{sde}, which cannot be gained generally, especially for nonlinear systems. In order to address  this issue, we go further and show that it is the uniquely twice continuously differentiable solution to a system of elliptic partial differential equations (e.g., \cite{gilbarg2015elliptic}) if such a solution exists, as formulated in Theorem \ref{eq1} below. It is worth noting here that the derived system of equations is different from Hamilton-Jacobi equations in \cite{Koutsoukos2008} or the variational inequalities in \cite{bujorianu2007new}.
\begin{theorem}
\label{eq1}
If there exist $v(\bm{x})\in \mathcal{C}^2(\overline{\mathcal{X}})$ and $u(\bm{x})\in \mathcal{C}^2(\overline{\mathcal{X}})$ such that  for $\bm{x}\in \overline{\mathcal{X}}$, 
\begin{align}
&\mathcal{A}v(\bm{x})=0,\label{con1}\\
&v(\bm{x})=1_{\mathcal{T}}(\bm{x})+\mathcal{A}u(\bm{x}),\label{con2}
\end{align}
then \[v(\bm{x})=V(\bm{x}), \forall \bm{x}\in \overline{\mathcal{X}}\] and thus \[{\rm RA}_p=\{\bm{x}\in \overline{\mathcal{X}}\mid v(\bm{x})> p\},\] where $V(\cdot): \overline{\mathcal{X}}\rightarrow \mathbb{R}$ is the value function in \eqref{value}. 
\end{theorem}
\begin{proof}
Let $\bm{x}\in \overline{\mathcal{X}}$. According to Proposition \ref{three}, we have that \[\widehat{\bm{X}}^{\bm{x}}(t,\bm{w})\in \overline{\mathcal{X}}\] for $t \in \mathbb{R}_{\geq 0}$ and $\bm{w}\in \Omega$.

From Eq. \eqref{con1} and Theorem \ref{dyn_theom}, together with the fact that any constant time $t\in \mathbb{R}_{\geq 0}$ is a stopping time with $E[t]<\infty$, we have that 
\begin{equation}
\label{con11}
v(\bm{x})=E[v(\widehat{\bm{X}}^{\bm{x}}(t,\bm{w}))], \forall  t \in \mathbb{R}_{\geq 0}.
\end{equation}

From Eq. \eqref{con2}, we have that for $t\in \mathbb{R}_{\geq 0}$,
\begin{equation}
\label{con21}
v(\widehat{\bm{X}}^{\bm{x}}(t,\bm{w}))=1_{\mathcal{T}}(\widehat{\bm{X}}^{\bm{x}}(t,\bm{w}))+\mathcal{A}u(\widehat{\bm{X}}^{\bm{x}}(t,\bm{w})).
\end{equation}

From Eq. \eqref{con21}, we have that for $t\in \mathbb{R}_{\geq 0}$,
\begin{equation*}
\begin{split}
E[\int_{0}^t v(\widehat{\bm{X}}^{\bm{x}}(s,\bm{w})) ds]&=E[\int_{0}^t 1_{\mathcal{T}}(\widehat{\bm{X}}^{\bm{x}}(s,\bm{w})) ds]\\
&+E[\int_{0}^t \mathcal{A}u(\widehat{\bm{X}}^{\bm{x}}(s,\bm{w})) ds]
\end{split}
\end{equation*}
and further according to Fubini's theorem \cite{rosenthal2006first},
\begin{equation*}
\begin{split}
\int_{0}^tE[ v(\widehat{\bm{X}}^{\bm{x}}(s,\bm{w}))] ds&=E[\int_{0}^t 1_{\mathcal{T}}(\widehat{\bm{X}}^{\bm{x}}(s,\bm{w})) ds]\\
&+E[\int_{0}^t \mathcal{A}u(\widehat{\bm{X}}^{\bm{x}}(s,\bm{w})) ds].
\end{split}
\end{equation*}
Consequently, 
\[ 
\begin{split}
v(\bm{x})=&\frac{E[\int_{0}^{t} 1_{\mathcal{T}}(\widehat{\bm{X}}^{\bm{x}}(s,\bm{w})) ds]}{t}\\
&+\frac{E[u(\widehat{\bm{X}}^{\bm{x}} (t,\bm{w}))]-u(\bm{x})}{t}, \forall t\in \mathbb{R}_{\geq 0}.
\end{split}
\]
Since $u(\bm{x})$ is continuously differentiable over $\overline{\mathcal{X}}$, it is bounded over $\bm{x}\in \overline{\mathcal{X}}$. Consequently, 
\[v(\bm{x})=\lim_{t\rightarrow \infty} \frac{E[\int_{0}^{t} 1_{\mathcal{T}}(\widehat{\bm{X}}^{\bm{x}}(s,\bm{w}))ds]}{t}\]
and thus $v(\bm{x})=V(\bm{x})$,  implying further that \[{\rm RA}_p=\{\bm{x}\in \overline{\mathcal{X}} \mid v(\bm{x})>p\}\] from Lemma  \ref{value_reach}.
\end{proof}

\begin{remark}
\label{boundary}
If the boundary condition is known, i.e., the values for both functions $v(\bm{x})$ and $u(\bm{x})$ on the boundary $\partial \overline{\mathcal{X}}$ are known, the problem of solving the system of elliptic partial differential equations \eqref{con1} and \eqref{con2} is a known Dirichlet problem, which originally was posed for Laplace's equation and could be solved using existing methods such as the Perron's method \cite{ishii1989uniqueness}. Unfortunately, we only know the values of the function $v(\bm{x})$ on the boundary $\partial \overline{\mathcal{X}}$, i.e., $v(\bm{x})=0$ for $\bm{x}\in \partial \overline{\mathcal{X}}$, which can be gained from equation \eqref{con2}, the values of the function $u(\bm{x})$ on $\partial \overline{\mathcal{X}}$ is unknown. Besides, even if the boundary condition is known, in order to characterize the $p$-reach-avoid set we need to guarantee that the obtained solution $(v(\bm{x}),u(\bm{x}))$ to the system of equations \eqref{con1} and \eqref{con2} is twice continuously differentiable solution from Theorem \ref{eq1}. Whether $(v(\bm{x}),u(\bm{x}))$ otherwise, i.e., if not twice continuously differentiable, can be used to characterize the p-reach-avoid remains currently open. \qed
\end{remark}

\begin{remark}
 \label{classical_com}
Let's present another result related to the reach probability of leaving the set $\mathcal{X}\setminus \mathcal{T}$ through $\mathcal{T}$. This result can be obtained via Proposition 7.2 in \cite{karatzas2012brownian}.   

\textcolor{red}{Let $v$ be a solution of the Dirichlet problem in the open, bounded domain $\mathcal{X}\setminus \mathcal{T}$:
    \begin{equation}
    \label{classical}
        \begin{split}
        &\mathcal{A} v(\bm{x})=0, \text{~in~}\mathcal{X}\setminus \mathcal{T},\\
        &v(\bm{x})=1, \text{~in~} \mathcal{T}, \\
        &v(\bm{x})=0, \text{~in~}\partial \mathcal{X},
        \end{split}
    \end{equation}
    and let $\widehat{\tau}^{\bm{x}}(\bm{w})=\inf\{t\geq 0; \widehat{\bm{X}}^{\bm{x}}(t,\bm{w})\in \partial \mathcal{T}\cup \partial \mathcal{X}\}$. If 
    \begin{equation}
        E[\widehat{\tau}^{\bm{x}}(\bm{w})]<\infty, \forall \bm{x}\in \mathcal{X}\setminus \mathcal{T}, 
    \end{equation}
    then we have
    \[v(\bm{x})=E[1_{\partial \mathcal{T}}(\widehat{\bm{X}}^{\bm{x}}(\widehat{\tau}^{\bm{x}}(\bm{w}),\bm{w}))], \bm{x}\in \mathcal{X}\setminus \mathcal{T},\]
which is the probability of leaving $\mathcal{X}\setminus \mathcal{T}$ through $\mathcal{T}$ for the process $\widehat{\bm{X}}^{\bm{x}}(t,\bm{w})$. }

\textcolor{red}{Under the condition that 
\[ E[\widehat{\tau}^{\bm{x}}(\bm{w})]<\infty, \forall \bm{x}\in \mathcal{X}\setminus \mathcal{T}, \] which implies that $\mathbb{P}(\widehat{\tau}^{\bm{x}}(\bm{w})<\infty)=1, \forall \bm{x}\in \mathcal{X}\setminus \mathcal{T}$, we have $v(\bm{x})=V(\bm{x}), \forall \bm{x}\in \mathcal{X}\setminus \mathcal{T}$ and thus
\[\{\bm{x}\in \overline{\mathcal{X}}\mid v(\bm{x})>p\}={\rm RA}_p.\]  However, when there exists $\bm{x}_0 \in \mathcal{X}$ such that $E[\widehat{\tau}^{\bm{x}_0}(\bm{w})]=\infty$, we cannot obtain that \[v(\bm{x}_0)=V(\bm{x}_0).\] Consequently, 
$\{\bm{x}\in \overline{\mathcal{X}}\mid v(\bm{x})>p\}={\rm RA}_p$ may not hold. $\{\bm{x}\in \overline{\mathcal{X}}\mid v(\bm{x})>p\}$ may include states such that system \eqref{sde} starting from them will stay inside $\mathcal{X}\setminus \mathcal{T}$ for all the time with probability being larger than $p$, but will reach $\mathcal{T}$ in finite time with probability being smaller than or equal to $p$.  }

\textcolor{red}{In contrast, when relating equations \eqref{con1} and \eqref{con2} to the set ${\rm RA}_p$, Theorem \ref{eq1} does not impose the condition that  $E[\widehat{\tau}^{\bm{x}}(\bm{w})]<\infty, \forall  \bm{x}\in \mathcal{X}\setminus \mathcal{T}$. Therefore, Theorem \ref{eq1} applies to the case with $E[\widehat{\tau}^{\bm{x}}(\bm{w})]=\infty$ and thus is more general. Let's further compare 
equations \eqref{con1} and \eqref{con2} with \eqref{classical}. We first reformulate equations \eqref{con1} and \eqref{con2} in the following equivalent form  
\begin{equation}
    \label{equiv_a}
    \begin{split}
    &\mathcal{A}v(\bm{x})=0, \text{in~}\mathcal{X}\setminus \mathcal{T},\\
    &v(\bm{x})=1, \text{in~}\mathcal{T},\\
    &v(\bm{x})=0, \text{in~}\partial \mathcal{X},\\
    &v(\bm{x})=\mathcal{A}u(\bm{x}), \text{in~}\mathcal{X}\setminus \mathcal{T}.
    \end{split}
\end{equation}
It is easy to find that equations \eqref{con1} and \eqref{con2} have an additional constraint \[v(\bm{x})=\mathcal{A}u(\bm{x}), \text{in~}\mathcal{X}\setminus \mathcal{T},\] comparing with \eqref{classical}. If $E[\widehat{\tau}^{\bm{x}}(\bm{w})]<\infty, \forall \bm{x}\in \mathcal{X}\setminus \mathcal{T}$, this constraint is redundant and can be removed, thus turning equations \eqref{equiv_a} into \eqref{classical}. In this case, according to Dynkin's formula in Theorem \ref{dyn_theom}, $u(\bm{x})$ can take  \[u(\bm{x})=E[u(\widehat{\bm{X}}^{\bm{x}}(\widehat{\tau}^{\bm{x}}(\bm{w}),\bm{w})))-\int_{0}^ {\widehat{\tau}^{\bm{x}}(\bm{w})} v(\widehat{\bm{X}}^{\bm{x}}(s,\bm{w}))) ds].\]
Otherwise, this constraint cannot be removed and its existence ensures that 
$\{\bm{x}\in \overline{\mathcal{X}}\in v(\bm{x})>p\}={\rm RA}_p$.
} 
\qed
\end{remark}

From Theorem \ref{eq1} we have that if we obtain a twice continuously differentiable solution $(v(\bm{x}),u(\bm{x}))$ to equations \eqref{con1} and \eqref{con2} the exact $p$-reach-avoid set ${\rm RA}_p=\{\bm{x}\in \mathcal{X}\mid v(\bm{x})>p\}$ can be gained. However, due to the existence of the indicator function $1_{\mathcal{T}}(\bm{x})$ in \eqref{con2}, we have that $\lim_{\bm{x}\rightarrow \bm{x}_0, \bm{x}\notin \mathcal{T}}v(\bm{x})=\mathcal{A}u(\bm{x}_0)$, which is not equal to $v(\bm{x}_0)=1+\mathcal{A}u(\bm{x}_0)$, where $\bm{x}_0 \in \partial \mathcal{T}$, thus the system of equations \eqref{con1} and \eqref{con2} does not admit twice continuously differentiable solutions $(v(\bm{x}),\bm{u}(\bm{x}))$ generally. Let's take an extreme case: $\bm{b}(\bm{x})=\bm{0}$ for $\bm{x}\in \mathbb{R}^n$ and $\bm{\sigma}(\bm{x})=0$ for $\bm{x}\in \mathbb{R}^n$. We can obtain that if $v(\bm{x})$ is a solution to the system of equations \eqref{con1} and \eqref{con2}, then $v(\bm{x})=1$ for $\bm{x}\in \mathcal{T}$ and $v(\bm{x})=0$ for $\bm{x}\in \overline{\mathcal{X}}\setminus \mathcal{T}$, which implies that $v(\bm{x})\notin \mathcal{C}^2(\overline{\mathcal{X}})$. Despite all this, equations \eqref{con1} and \eqref{con2} play a fundamental role in our method for inner-approximating  the $p$-reach-avoid set. In the subsequent subsection we will introduce a semi-definite programming based method for inner-approximating the $p$-reach-avoid set, which is obtained by relaxing the equations \eqref{con1} and \eqref{con2} into a system of inequalities.

\subsection{Semi-definite Programming Implementation}
\label{SDPI}
In this subsection a semi-definite programming method is presented for inner-approximating the $p$-reach-avoid set ${\rm RA}_p$. 

First, we observe that an inner-approximation could be obtained via solving a system of inequalities, which is derived from Eq. \eqref{con1} and \eqref{con2}.
\begin{corollary}
\label{inequality}
If there exist functions $v(\bm{x})\in \mathcal{C}^2(\overline{\mathcal{X}})$ and $u(\bm{x})\in \mathcal{C}^2(\overline{\mathcal{X}})$ such that for $\bm{x}\in \overline{\mathcal{X}}$, 
\begin{align}
&\mathcal{A}v(\bm{x})\geq 0,\label{con3}\\
&1_{\mathcal{T}}(\bm{x})+\mathcal{A}u(\bm{x})\geq v(\bm{x}),\label{con4}
\end{align}
then \[\{\bm{x}\in \overline{\mathcal{X}}\mid v(\bm{x})> p\}\subseteq {\rm RA}_p\] is an inner-approximation of the $p$-reach-avoid set ${\rm RA}_p$.
\end{corollary}
\begin{proof}
According to Proposition \ref{three}, 
\[\widehat{\bm{X}}^{\bm{x}}(\tau,\bm{w}) \in \overline{\mathcal{X}}, \forall \tau \in \mathbb{R}_{\geq 0}, \forall\bm{w}\in \Omega,\] if $\bm{x}\in \overline{\mathcal{X}}$.  

Let $\bm{x}_0\in \{\bm{x}\in \overline{\mathcal{X}}\mid v(\bm{x})> p\}
$. From constraint \eqref{con3}, we have that 
\begin{equation}
\label{leq}
v(\bm{x}_0)\leq E[v(\widehat{\bm{X}}^{\bm{x}_0}(t,\bm{w}))], \forall t\in \mathbb{R}_{\geq 0}.
\end{equation}
Also, constraint \eqref{con4} indicates that 
\begin{equation*}
\begin{split}
v(\widehat{\bm{X}}^{\bm{x}_0}(t,\bm{w}))\leq 1_{\mathcal{T}}(\widehat{\bm{X}}^{\bm{x}_0}(t,\bm{w}))+\mathcal{A}u(\widehat{\bm{X}}^{\bm{x}_0}(t,\bm{w}))
\end{split}
\end{equation*}
holds for $t\in \mathbb{R}_{\geq 0}$ and $\bm{w}\in \Omega$.
Thus, we have that  for $t\in \mathbb{R}_{\geq 0}$,
\begin{equation*}
\begin{split}
E[\int_{0}^t v(\widehat{\bm{X}}^{\bm{x}_0}(\tau,\bm{w}))d \tau]&\leq E[\int_{0}^t 1_{\mathcal{T}}(\widehat{\bm{X}}^{\bm{x}_0}(\tau,\bm{w}))d\tau]\\
&+ E[\int_{0}^t \mathcal{A}u(\widehat{\bm{X}}^{\bm{x}_0}(\tau,\bm{w})) d\tau]
\end{split}
\end{equation*}
and thus 
\[
\begin{split}
&\int_{0}^t E[v(\widehat{\bm{X}}^{\bm{x}_0}(\tau,\bm{w}))]d \tau
\\
&\leq E[\int_{0}^t 1_{\mathcal{T}}(\widehat{\bm{X}}^{\bm{x}_0}(\tau,\bm{w}))d\tau]+ E[u(\widehat{\bm{X}}^{\bm{x}_0}(t,\bm{w}))]-u(\bm{x}_0).
\end{split}
\]

Combining with \eqref{leq} we further have that 
\[
\begin{split}
v(\bm{x}_0)&\leq \frac{E[\int_{0}^t 1_{\mathcal{T}}(\widehat{\bm{X}}^{\bm{x}_0}(\tau,\bm{w}))d\tau]}{t}\\
&+ \frac{E[u(\widehat{\bm{X}}^{\bm{x}_0}(t,\bm{w}))]-u(\bm{x}_0)}{t} , \forall t\in \mathbb{R}_{\geq 0}
\end{split}
\]
and thus 
\begin{equation}
\label{ineqa11}
\begin{split}
v(\bm{x}_0)&\leq \lim_{t\rightarrow \infty}\frac{E[\int_{0}^t 1_{\mathcal{T}}(\widehat{\bm{X}}^{\bm{x}_0}(\tau,\bm{w}))d\tau]}{t}~~~~(=V(\bm{x}_0))\\
&+ \lim_{t\rightarrow \infty}\frac{E[u(\widehat{\bm{X}}^{\bm{x}_0}(t,\bm{w}))]-u(\bm{x}_0)}{t}.
\end{split}
\end{equation}

Since $ \lim_{t\rightarrow \infty}\frac{E[u(\widehat{\bm{X}}^{\bm{x}_0}(t,\bm{w}))]-u(\bm{x}_0)}{t}=0$, we have \[p<v(\bm{x}_0)\leq V(\bm{x}_0).\] Also, since constraint \eqref{con4} indicates that \[v(\bm{x})\leq 0\] for $\bm{x}\in \overline{\mathcal{X}}\setminus \mathcal{X}$, $\bm{x}_0\in \mathcal{X}$ holds.
Consequently, $\{\bm{x}\in \overline{\mathcal{X}}\mid v(\bm{x})> p\} \subseteq {\rm RA}_p$.
\end{proof}

\begin{remark}
Although the equations \eqref{con1} and \eqref{con2} do not admit twice continuously differentiable solutions in general, the system of inequalities \eqref{con3} and \eqref{con4} does. The pair that $v(\bm{x})\equiv 0$ and $u(\bm{x})\equiv 0$ for $\bm{x}\in \overline{\mathcal{X}}$ satisfies the system of inequalities \eqref{con3} and \eqref{con4}. \qed
\end{remark}

Corollary \ref{inequality} expresses that an inner-approximation of the $p$-reach-avoid set ${\rm RA}_p$ is provided by a solution $v(\bm{x})\in \mathcal{C}^2(\overline{\mathcal{X}})$ to constraints \eqref{con3} and \eqref{con4}. Below we present a convex programming method for solving constraints \eqref{con3} and \eqref{con4}.

The equivalent constraints without indicator functions of constraints \eqref{con3} and \eqref{con4} are formulated below:
\begin{equation}
\label{upper2}
\begin{split}
&\big[\mathcal{A}v(\bm{x})\geq 0, \forall \bm{x}\in \mathcal{X}\setminus \mathcal{T}\big]\bigwedge \\
&\big[-v(\bm{x})+1_{\mathcal{T}}(\bm{x})+\mathcal{A} u(\bm{x})\geq 0, \forall \bm{x}\in \overline{\mathcal{X}}\big],
\end{split}
\end{equation}
which is  further equivalent to 
\begin{equation}
\label{upper3}
\begin{split}
&\mathcal{A}v(\bm{x}) \geq 0, \forall \bm{x}\in \mathcal{X}\setminus \mathcal{T}, \\
&-v(\bm{x})+\mathcal{A}u(\bm{x})\geq 0, \forall \bm{x}\in \mathcal{X}\setminus \mathcal{T}, \\
&-v(\bm{x})\geq 0, \forall \bm{x}\in \partial \mathcal{X}, \\
&1-v(\bm{x})\geq 0, \forall \bm{x}\in\mathcal{T}. 
\end{split}
\end{equation}

 If functions $v(\bm{x})$ and $u(\bm{x})$ in \eqref{upper3} are further restricted to polynomial functions over $\bm{x}\in \mathbb{R}^n$, we can encode the system of inequalities \eqref{upper3} in the form of sum-of-squares constraints, finally resulting in a semi-definite program \eqref{sos} for inner-approximating the $p$-reach-avoid set ${\rm RA}_p$. 
\begin{algorithm}
\begin{equation}
\label{sos}
\begin{split}
&\max  \bm{c} \cdot \hat{\bm{w}}\\
&\text{s.t.}\\
&\mathcal{A}v(\bm{x})+s_0(\bm{x}) h_0(\bm{x})+s_1(\bm{x}) (1-g(\bm{x}))\in \sum[\bm{x}],\\
&-v(\bm{x})+\mathcal{A}u(\bm{x})+s_2(\bm{x}) h_0(\bm{x})\\
&~~~~~~~~~~~~~~~~~~~~~~~~~~~~+s_3(\bm{x}) (1-g(\bm{x}))\in \sum[\bm{x}],\\
&-v(\bm{x})+p(\bm{x}) h_0(\bm{x}) \in \sum[\bm{x}],\\
&1-v(\bm{x})+s_4(\bm{x}) (g(\bm{x})-1) \in \sum[\bm{x}],
\end{split}
\end{equation}
where $\bm{c}\cdot \widehat{\bm{w}}=\int_{\overline{\mathcal{X}}}v(\bm{x})d\bm{x}$, $\widehat{\bm{w}}$ is the constant vector computed by integrating the monomials in $v(\bm{x})\in \mathbb{R}[\bm{x}]$ over $\overline{\mathcal{X}}$, $\bm{c}$ is the vector composed of unknown coffecients in $v(\bm{x})\in \mathbb{R}[\bm{x}]$; $u(\bm{x}), p(\bm{x})\in \mathbb{R}[\bm{x}]$ and $s_{i}(\bm{x})\in \sum[\bm{x}]$, $i=0,\ldots,4$.
\end{algorithm}

\begin{theorem}
\label{inner1}
Let $(v(\bm{x}),u(\bm{x}))$ be a solution to the semi-definite program \eqref{sos}, then \[\{\bm{x}\in \overline{\mathcal{X}}\mid v(\bm{x})> p\}\] is an inner-approximation of the $p$-reach-avoid set ${\rm RA}_p$.
\end{theorem}
\begin{proof}
Since $v(\bm{x})$ satisfies constrains in \eqref{sos} and $\partial \mathcal{X}\subseteq \{\bm{x}\in \mathbb{R}^n \mid h_0(\bm{x})=0\}$, we obtain that $v(\bm{x})$ satisfies \eqref{upper3}. Consequently, \[\{\bm{x}\in \overline{\mathcal{X}}\mid v(\bm{x})>p\}\subseteq {\rm RA}_p\] holds from Corollary \ref{inequality}.
\end{proof}

\begin{remark}
\label{determi}
If $\bm{\sigma}(\bm{x})\equiv \bm{0}$ for $\bm{x}\in \mathbb{R}^n$ in {\rm SDE} \eqref{sde}, {\rm SDE} \eqref{sde} is finally reduced to {\rm ODE} \eqref{ODE}:
\begin{equation}
\label{ODE}
\frac{d \bm{x}(t)}{d t}=\bm{b}(\bm{x}(t)),\bm{x}(0)=\bm{x}_0,
\end{equation}
whose solution is denoted by $\bm{X}^{\bm{x}_0}(\cdot): T\rightarrow \mathbb{R}^n$ with initial state $\bm{X}^{\bm{x}_0}(0)=\bm{x}_0$. In this case, if there exist functions $v(\bm{x})\in \mathcal{C}^1(\overline{\mathcal{X}})$ and $u(\bm{x})\in \mathcal{C}^1(\overline{\mathcal{X}})$ satisfying \eqref{upper3}, then
\[\{\bm{x}\in \overline{\mathcal{X}}\mid v(\bm{x})>0\} \subseteq {\rm RA},\]
where {\rm RA} is the reach-avoid set over open time horizons, i.e.,
\begin{equation}
\label{deter_reach}
{\rm RA}=\left\{\bm{x}_0\in \mathbb{R}^n\middle |\;
\begin{aligned}
&\exists t\in \mathbb{R}_{\geq 0}. \big[\bm{X}^{\bm{x}_0}(t)\in \mathcal{T}~~\bigwedge \\
&~~~~~~~~\forall \tau\in [0,t].\bm{X}^{\bm{x}_0}(\tau) \in \mathcal{X}\big]
\end{aligned}
\right\}.
\end{equation}
\begin{proof}
The conclusion can be obtained by following the proof of Corollary \ref{inequality} with small modifications. A brief explanation is given below.  

Taking $\bm{x}_0\in \{\bm{x}\in \overline{\mathcal{X}}\mid v(x)>0\}$ and following the proof of Corollary \ref{inequality} by removing the expectation operators, we obtain
\begin{equation*}
\begin{split}
v(\bm{x}_0)&\leq \lim_{t\rightarrow \infty}\frac{\int_{0}^t 1_{\mathcal{T}}(\widehat{\bm{X}}^{\bm{x}_0}(\tau))d\tau}{t}\\
&+ \lim_{t\rightarrow \infty}\frac{u(\widehat{\bm{X}}^{\bm{x}_0}(t))-u(\bm{x}_0)}{t}, ~~(\text{Corresponding to \eqref{ineqa11}})
\end{split}
\end{equation*}
where $\widehat{\bm{X}}^{\bm{x}_0}(t)=\bm{X}^{\bm{x}_0}(t)$ for $t\leq \widehat{\tau}_{\mathcal{T}}^{\bm{x_0}}$ and $\widehat{\bm{X}}^{\bm{x}_0}(t)=\bm{X}^{\bm{x}_0}(\widehat{\tau}_{\mathcal{T}}^{\bm{x_0}})$ for $t\geq \widehat{\tau}_{\mathcal{T}}^{\bm{x_0}}$ with $\widehat{\tau}_{\mathcal{T}}^{\bm{x_0}}=\inf\{t\mid \bm{X}^{\bm{x}_0}(t) \in \partial \mathcal{X} \bigvee \bm{X}^{\bm{x}_0}(t) \in \mathcal{T}\}$. 

Since $ \lim_{t\rightarrow \infty}\frac{u(\widehat{\bm{X}}^{\bm{x}_0}(t))-u(\bm{x}_0)}{t}=0$ and $v(\bm{x}_0)>0$, $\bm{x}_0\in {\rm RA}$ holds. Thus, $\{\bm{x}\in \overline{\mathcal{X}}\mid v(\bm{x})>0\} \subseteq {\rm RA}$.

Another proof can be found in Proposition 5 in \cite{xue2023reach}.
\end{proof}

\end{remark}

\begin{remark}
\label{UN}
When the target set $\mathcal{T}$ is an unsafe set and an initial set {\rm INI} $\subset \mathcal{X}$ is given, a set of constraints can also be constructed for addressing the classical safety verification problem of {\rm SDE} \eqref{sde} as in \cite{PrajnaJP07} via relaxing the equations \eqref{con1} and \eqref{con2}. That is, we can compute a probability $p$ such that for $\bm{x}_0\in {\rm INI}$,
\begin{equation*}
\begin{split}
&\mathbb{P}\Big(\exists t \in  \mathbb{R}_{\geq 0}.\big[ \bm{X}^{\bm{x}_0}(t,\bm{w})\in \mathcal{T}\bigwedge \forall \tau\in [0,t]. \bm{X}^{\bm{x}_0}(\tau,\bm{w})\in \mathcal{X}\big]\Big)\\
&\leq p.
\end{split}
\end{equation*}

This method is orthogonal to stochastic barrier-certificate methods from \cite{PrajnaJP07}, since our method is derived from the equations \eqref{con1} and \eqref{con2} rather than Doob's martingale inequality as in \cite{PrajnaJP07}. We will compare them in the future work.
\begin{corollary}
\label{inequality1}
If there exist functions $v(\bm{x})\in \mathcal{C}^2(\overline{\mathcal{X}})$ and $u(\bm{x})\in \mathcal{C}^2(\overline{\mathcal{X}})$ such that for $\bm{x}\in \overline{\mathcal{X}}$, 
\begin{align}
&-\mathcal{A}v(\bm{x})\geq 0,\label{con31}\\
&v(\bm{x})\geq 1_{\mathcal{T}}(\bm{x})+\mathcal{A}u(\bm{x}),\label{con41}
\end{align}
then \[\{\bm{x}\in \overline{\mathcal{X}}\mid v(\bm{x})\leq p\}\subseteq {\rm RA'_p}\] is an inner-approximation of the p-reach-avoid set ${\rm RA'}_p$, 
where 
\[
{\rm RA'}_p=\left\{\bm{x} \in \mathcal{X}\middle|\;
\begin{aligned}
&\mathbb{P}\Big(\exists t \in  \mathbb{R}_{\geq 0}. [\bm{X}^{\bm{x}}(t,\bm{w})\in \mathcal{T}~\bigwedge\\
&~~~~\forall \tau\in [0,t]. \bm{X}^{\bm{x}}(\tau,\bm{w})\in \mathcal{X}]\Big)\leq p
\end{aligned}
\right\}.
\]
\end{corollary}
\begin{proof}
The conclusion can be obtained by following the arguments for Corollary \ref{inequality}.
\end{proof}

Obviously, if there exists $p\in [0,1)$ such that \[{\rm INI}\subseteq \{\bm{x}\in \overline{\mathcal{X}}\mid v(\bm{x})\leq p\},\] where $v(\bm{x})$ satisfies Corollary \ref{inequality1}, then 
\begin{equation*}
\begin{split}
   & \mathbb{P}\Big(\exists t \in  \mathbb{R}_{\geq 0}.\big[ \bm{X}^{\bm{x}_0}(t,\bm{w})\in \mathcal{T}\bigwedge \forall \tau\in [0,t]. \bm{X}^{\bm{x}_0}(\tau,\bm{w})\in \mathcal{X}\big]\Big)\\
   &\leq p
\end{split}
\end{equation*}
holds for $\bm{x}_0\in {\rm INI}$.

According to Corollary \ref{inequality1}, the safety verification problem can be encoded into the problem of solving the following constraints:
\begin{equation}
\label{safety_verification}
   \begin{split}
   &v(\bm{x})\leq p, \forall \bm{x}\in {\rm INI},\\
    &v(\bm{x})\geq 1, \forall \bm{x}\in \mathcal{T},\\
       &v(\bm{x})\geq \mathcal{A}u(\bm{x}), \forall \bm{x}\in \mathcal{X}\setminus \mathcal{T},\\
       &\mathcal{A}v(\bm{x})\leq 0, \forall \bm{x}\in \mathcal{X},\\
       &v(\bm{x})\geq 0, \forall \bm{x}\in \partial \mathcal{X}.
   \end{split} 
\end{equation}
Comparing the set of constraints \eqref{safety_verification} and constraints (26)-(29) in \cite{PrajnaJP07}, the main difference between them lies in that the former uses the constraint  $v(\bm{x})\geq \mathcal{A}u(\bm{x}), \forall \bm{x}\in \mathcal{X}\setminus \mathcal{T}$ rather than 
$v(\bm{x})\geq 0, \forall \bm{x}\in \mathcal{X}$(It implies $v(\bm{x})\geq 0, \forall \bm{x}\in \partial \mathcal{X}$). Therefore, the set of constraints \eqref{safety_verification} is more expressive than constraints (26)-(29) in \cite{PrajnaJP07}, which is a special instance of the set of constraints \eqref{safety_verification} with $u(\bm{x})\equiv 0$ for $\bm{x}\in \mathcal{X}$. \qed

\end{remark}
\section{Examples}
\label{sec:exam}
In this section we demonstrate on several examples the performance of our approach exploiting semi-definite programming. All computations solving \eqref{sos} were performed on an i7-7500U 2.70GHz CPU with 32GB RAM running Windows 10, where the sum-of-squares module of YALMIP \cite{lofberg2004} was used to transform the sum-of-squares optimization problem \eqref{sos} into a semi-definite program and the solver Mosek \cite{mosek2015mosek} was used to solve the resulting semi-definite program. The parameters controlling the performance of our semi-definite programming approach are presented in Table \ref{table1}. 

\begin{table}
\begin{center}
\begin{tabular}{|c|c|c|c|c|c|c|c|}
  \hline
  \multirow{1}{*}{}&\multicolumn{5}{|c|}{\texttt{SDP} \eqref{sos}} \\\hline
    Ex.&$d_{v}$&$d_{u}$&$d_{s}$&$d_p$&$T$\\\hline
   1    &8&8&8&8&1.78\\\hline
   2&   16 &16&16&16 &4.39\\\hline
   3    &16&16&16&16 &3.78\\\hline
   4   &18&18&18&10 &6.11\\\hline
    5  &20&20&20&20 &8.27 \\\hline
   \end{tabular}\end{center}
\caption{Parameters of our implementations on \eqref{sos} for Examples \ref{ex1}$\sim$\ref{ship}. $d_{v}$, $d_{u}$ and $d_p$: degree of polynomials $v(\bm{x})$, $u(\bm{x})$ and $p(\bm{x})$ in \eqref{sos}, respectively; $d_{s}:$ degree of polynomials $s_i$ in \eqref{sos}, respectively, $i=0,\ldots,4$; $T$: computation time (seconds).}
\label{table1}
\end{table}

\begin{example}[Population growth]
\label{ex1}
Consider the stochastic dynamical system
\[d X(t,w)=b(X(t,w))dt+\sigma(X(t,w))d W(t,w),\]
with $b(X(t,w))=-X(t,w)$ and $\sigma(X(t,w))=\frac{\sqrt{2}}{2} X(t,w)$, which is a stochastic model of population dynamics subject to random fluctuations that can be attributed to extraneous or chance factors such as the weather, location, and the general environment.

Suppose that the safe set is $\mathcal{X}=\{x\in \mathbb{R}\mid x^2-1<0\}$ and the target set is $\mathcal{T}=\{x \in \mathbb{R}\mid 100x^2 \leq 1\}$.

The computed inner-approximations of 0.9- and 0.5-reach-avoid sets are respectively illustrated in Fig. \ref{fig-zero1} and \ref{fig-zero2}, which also shows the computed function $v(x)$ via solving the semi-definite program \eqref{sos}. For gauging the quality of the computed inner-approximations, the 0.9- and 0.5-reach-avoid sets estimated via Monto-Carlo methods are also respectively presented in Fig. \ref{fig-zero1} and \ref{fig-zero2} for comparisons.

\begin{figure}
\center
\includegraphics[width=2.5in,height=1.5in]{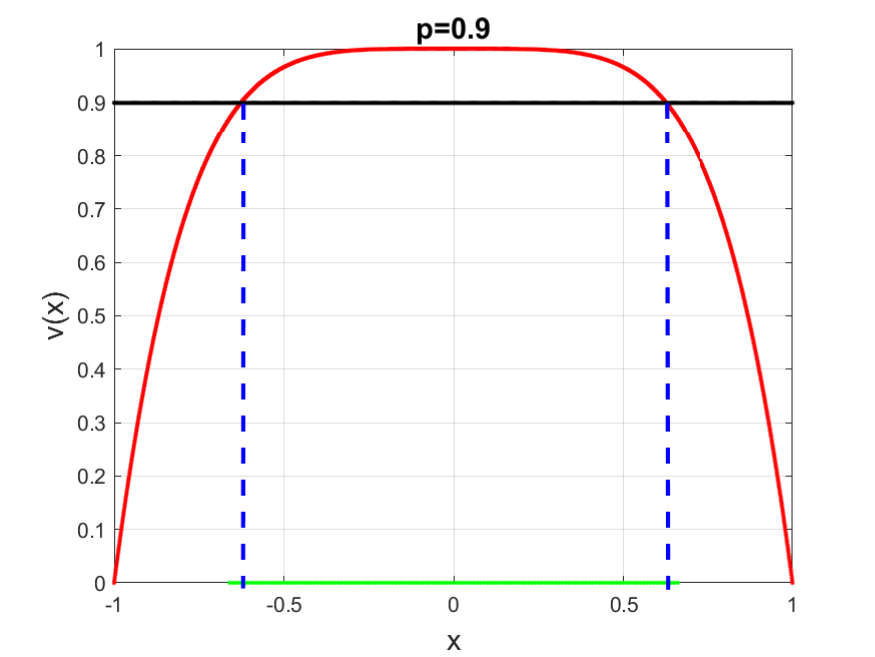} 
\caption{An illustration of inner-approximations of the reach-avoid set ${\rm RA}_p$ for Example \ref{ex1}. (
Red curve denotes level sets of the function $v(x)$ computed by solving \eqref{sos}. The set of states between the two dashed blue lines is an inner-approximation of the 0.9-reach-avoid set. The set of green states is the 0.9-reach-avoid set estimated via Monte-Carlo methods.)}
\label{fig-zero1}
\end{figure}

\begin{figure}
\center
\includegraphics[width=2.5in,height=1.5in]{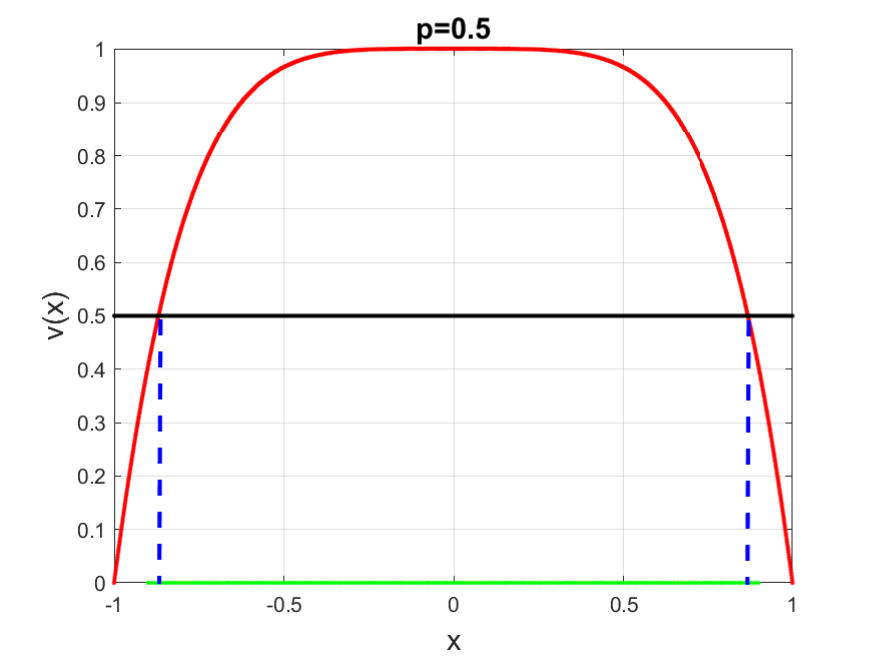} 
\caption{An illustration of inner-approximations of the reach-avoid set ${\rm RA}_p$ for Example \ref{ex1}. (
Red curve denotes level sets of the function $v(x)$ computed by solving \eqref{sos}. The set of states between the two dashed blue lines is an inner-approximation of the 0.5-reach-avoid set. The set of green states is the 0.5-reach-avoid set estimated via Monte-Carlo methods.)}
\label{fig-zero2}
\end{figure}

\end{example}

\begin{example}[Nonlinear drift]
\label{ex2}
Consider the nonlinear stochastic differential equation from \cite{prajna2004}, 
\begin{equation*}
\begin{split}
&dX_{1}(t,w)=X_{2}(t,w)dt,\\
&dX_{2}(t,w)=-(X_{1}(t,w)+X_{2}(t,w)+0.5X_{1}^3(t,w))dt\\
&~~~~~~~~~~~~~~~~~~~~~~~~~~~~~~~~~~~~~~~~~~~~~~~~~~+\sigma d W(t,w),
\end{split}
\end{equation*}
where $\sigma=0.1$. 

Suppose that the safe set and the target set are $\mathcal{X}=\{(x_1,x_2)^{\top}\in \mathbb{R}^2 \mid x_1^2+x_2^2-1<0\}$ and $\mathcal{T}=\{(x_1,x_2)^{\top}\in \mathbb{R}^2\mid 100(x_1-0.1)^2+100x_2^2\leq 1\}$, respectively. 

The computed function $v(x_1,x_2)$ via solving the semi-definite program \eqref{sos} is shown in Fig. \ref{fig-one_one} and the computed $0.5$- and $0.9$-reach-avoid sets are illustrated in Fig. \ref{fig-one}, which also shows two trajectories starting from $(0.5,0.5)^{\top}$ and $(-0.95,0.0)^{\top}$ respectively.  Also, we use the Monte-Carlo simulation method to assess the quality of computed inner-approximations, which is demonstrated in Fig. \ref{fig-one1}.

\begin{figure}
\center
\includegraphics[width=3.3in,height=2.5in]{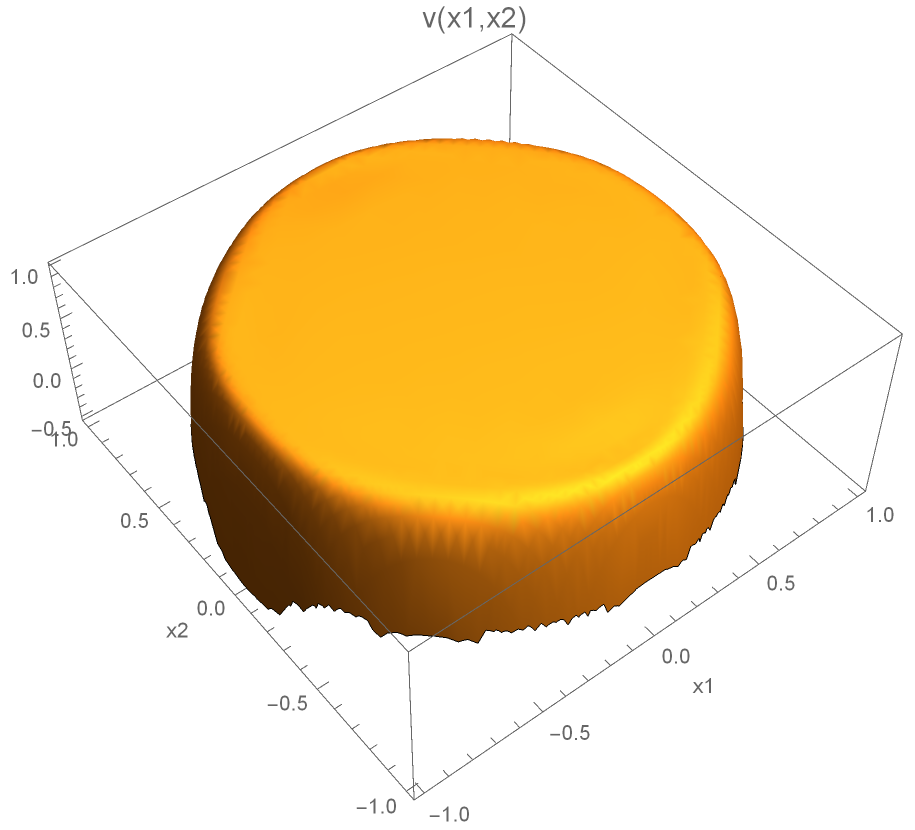} 
\caption{An illustration of the computed function $v(x_1,x_2)$ for Example \ref{ex2}.}
\label{fig-one_one}
\end{figure}

\begin{figure}
\center
\includegraphics[width=3.3in,height=1.8in]{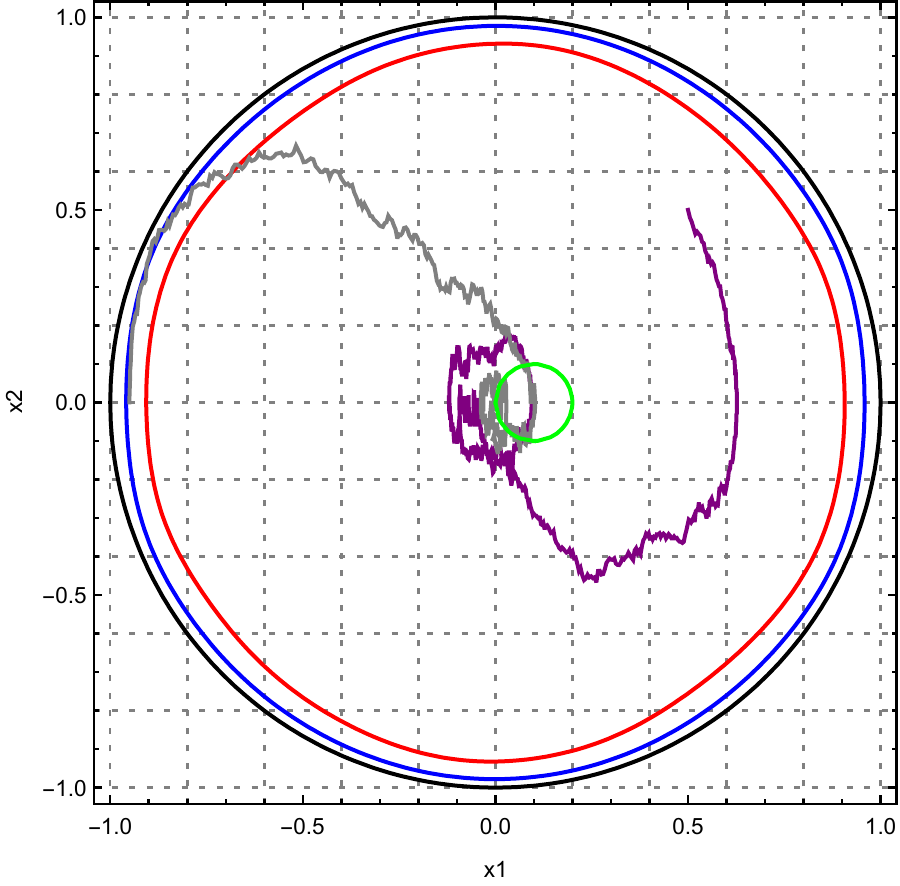} 
\caption{An illustration of inner-approximations of the reach-avoid set ${\rm RA}_p$ for Example \ref{ex2}. (Black and green curves denote the boundaries of the safe set $\mathcal{X}$ and the target set $\mathcal{T}$, respectively. Blue and red curves denote the boundaries of computed inner-approximation of the 0.5-reach-avoid set and 0.9-reach-avoid set, respectively.)}
\label{fig-one}
\end{figure}

\begin{figure}[htbp]
\center
\includegraphics[width=1.5in,height=1.8in]{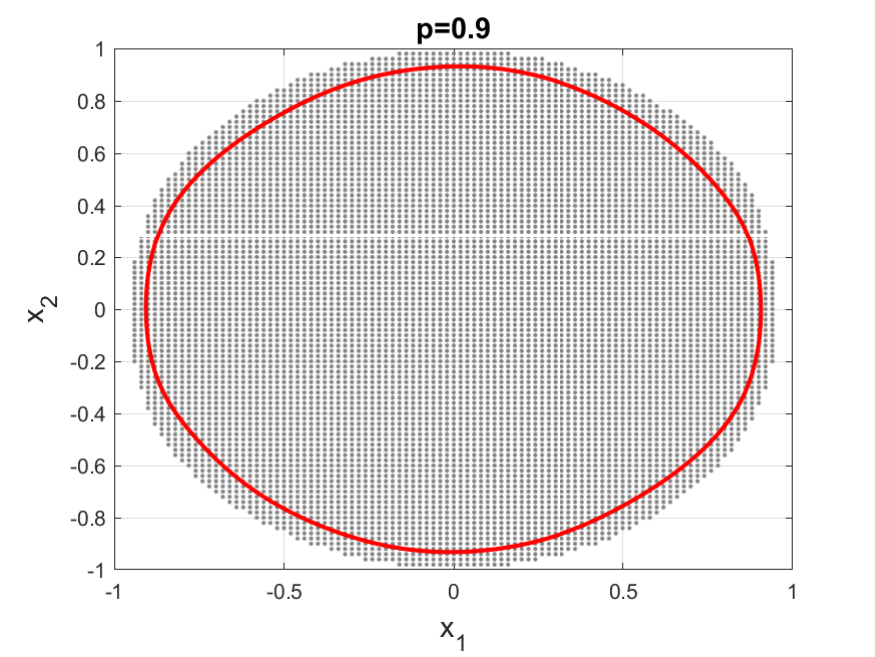} 
\includegraphics[width=1.5in,height=1.8in]{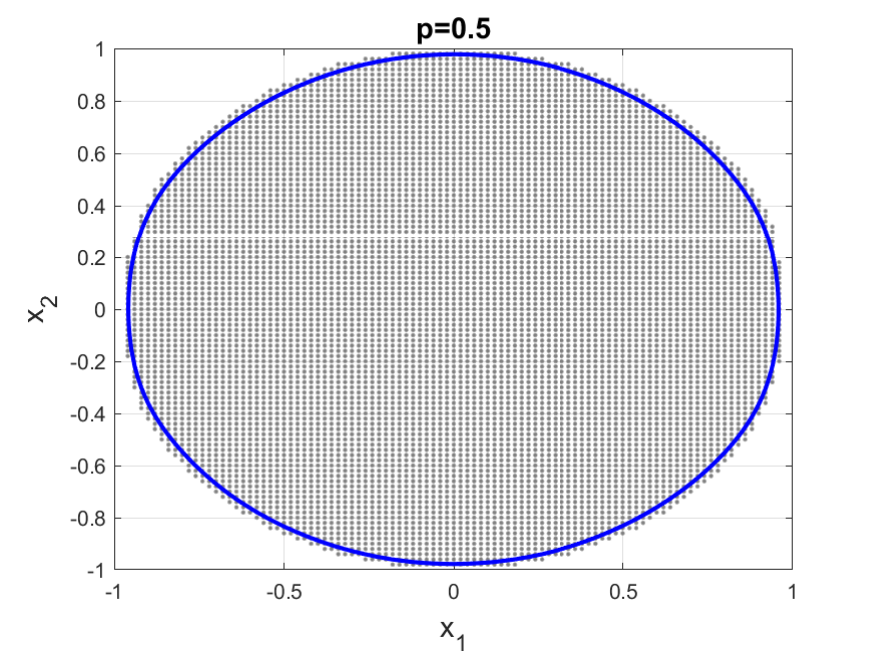}
\caption{An illustration of the quality of computed inner-approximations of the reach-avoid set ${\rm RA}_p$ for Example \ref{ex2}. (Red and blue curves denote the boundaries of computed inner-approximations of the 0.9- and 0.5-reach-avoid sets, respectively. Gray region denotes the $0.9/0.5$-reach-avoid set estimated via the Monte-Carlo simulation method.)}
\label{fig-one1}
\end{figure}

\end{example}

\begin{example}[Harmonic oscillator]
\label{ex3}
Consider a two-dimensional harmonic oscillator with noisy damping, 
\begin{equation*}
\begin{split}
&dX_{1}(t,w)=\zeta X_{2}(t,w)dt,\\
&dX_{2}(t,w)=(-\zeta X_{1}(t,w)-kX_{2}(t,w))dt\\
&~~~~~~~~~~~~~~~~~~~~~~~~~~~ -\sigma X_{2}(t,w) dW(t,w),
\end{split}
\end{equation*}
with $\zeta=1,k=7$ and $\sigma=2$. 

Suppose that the safe set and the target set are $\mathcal{X}=\{(x_1,x_2)^{\top}\in \mathbb{R}^2 \mid x_1^2+x_2^2-1<0\}$ and $\mathcal{T}=\{(x_1,x_2)^{\top}\in \mathbb{R}^2 \mid 10x_1^2+10x_2^2\leq 1\}$, respectively. 

The computed value function $v(x_1,x_2)$ via solving the semi-definite program \eqref{sos} is shown in Fig. \ref{fig-three_one} and the corresponding computed 0.5- and 0.9-reach-avoid sets are illustrated in Fig. \ref{fig-three}. Two trajectories starting from $(0.8,0.0)^{\top}$ and $(-0.9,0.0)^{\top}$ respectively are also illustrated in Fig. \ref{fig-three}. Also, we use the Monte-Carlo simulation method to assess the quality of computed inner-approximations, which is demonstrated in Fig. \ref{fig_three2}.

\begin{figure}
\center
\includegraphics[width=2.5in,height=2.0in]{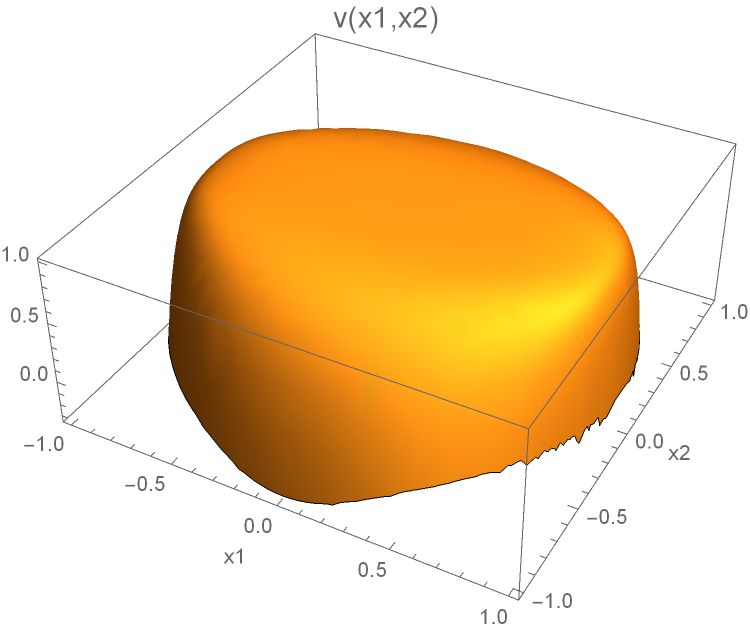} 
\caption{An illustration of the computed function $v(x_1,x_2)$ for Example \ref{ex3}.}
\label{fig-three_one}
\end{figure}

\begin{figure}
\center
\includegraphics[width=2.5in,height=1.8in]{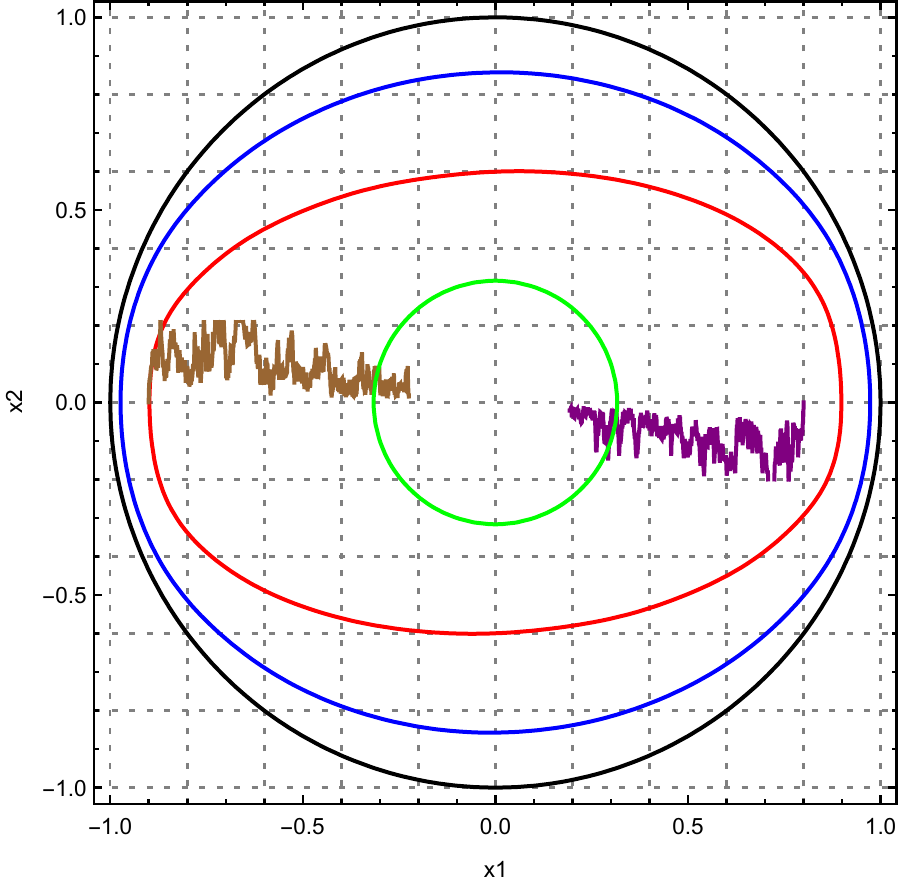} 
\caption{An illustration of computed inner-approximations for Example \ref{ex3}. (Black and green curves denote the boundaries of safe set $\mathcal{X}$ and target set $\mathcal{T}$, respectively. Red and blue curves denote the boundaries of computed inner-approximations of the 0.9- and 0.5-reach-avoid sets, respectively.)}
\label{fig-three}
\end{figure}

\begin{figure}[htbp]
\centering
\includegraphics[width=1.5in,height=1.8in]{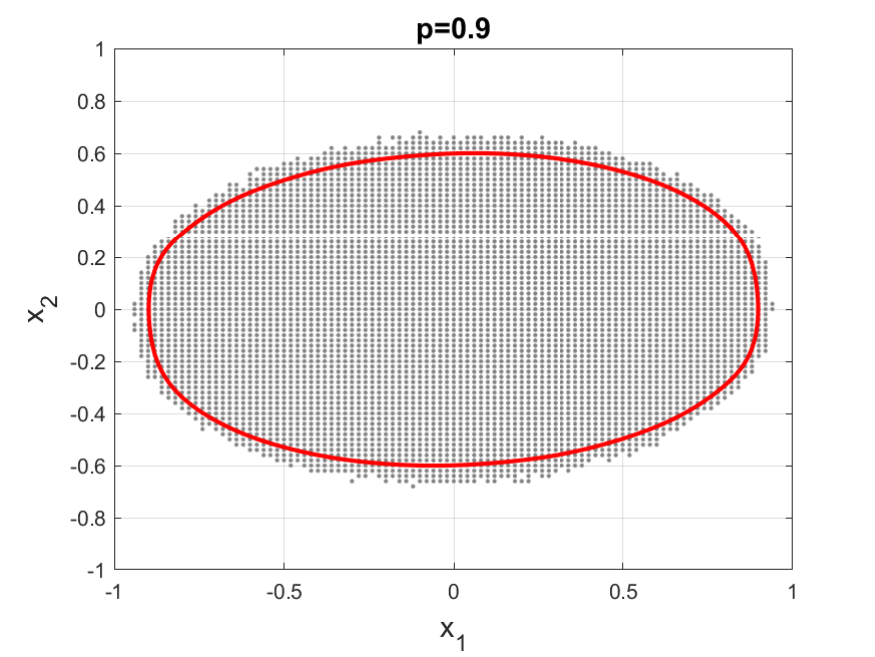} 
\includegraphics[width=1.5in,height=1.8in]{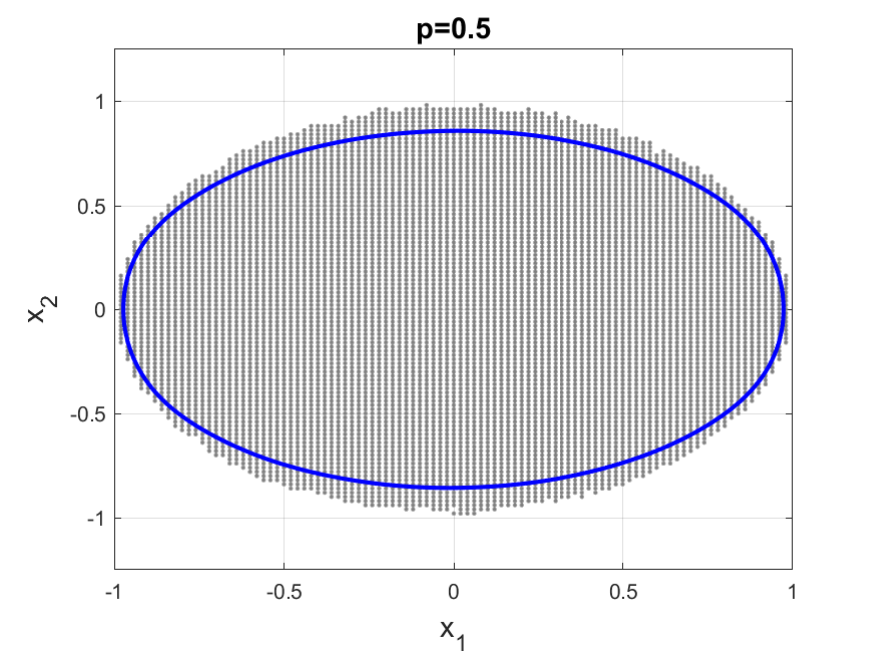} 
\caption{An illustration of the quality of computed inner-approximations of the reach-avoid set ${\rm RA}_p$ for Example \ref{ex3}. (Red and blue curves denote the boundaries of computed inner-approximations of the 0.9- and 0.5-reach-avoid sets, respectively. Gray region denotes the $0.9/0.5$-reach-avoid set estimated via the Monte-Carlo simulation method.)}
\label{fig_three2}
\end{figure}
\end{example}

\begin{example}
\label{ex5}
Consider the following nonlinear stochastic system from \cite{sloth2015safety},
\begin{equation*}
\begin{split}
&d X_{1}(t,\bm{w})=(-2X_{1}(t,\bm{w})+X^2_{2}(t,\bm{w}))dt+2 d W_{1}(t,\bm{w})\\
&d X_{2}(t,\bm{w})=-X_{2}(t,\bm{w}) dt+2 dW_{2}(t,\bm{w}).
\end{split}
\end{equation*} 

Suppose that the safe set and the target set are $\mathcal{X}=\{(x_1,x_2)^{\top}\in \mathbb{R}^2\mid x_1^2+x_2^2-1<0\}$ and $\mathcal{T}=\{(x_1,x_2)^{\top}\in \mathbb{R}^2\mid 4x_1^2+5(x_2-0.5)^2\leq 1\}$, respectively.

The computed function $v(x_1,x_2)$ via solving the semi-definite program \eqref{sos} is shown in Fig. \ref{fig-five_one} and the corresponding computed 0.1- and  0.5-reach-avoid sets are illustrated in Fig. \ref{fig-five1}. Two trajectories starting from $(-0.1,-0.5)^{\top}$ and $(0.1,0.5)^{\top}$ respectively are also illustrated in Fig. \ref{fig-five1}, one of which leaves the safe set $\mathcal{X}$. Also, we use the Monte-Carlo simulation method to assess the quality of computed inner-approximations, which is demonstrated in Fig. \ref{fig_five2}.

\begin{figure}
\center
\includegraphics[width=3.3in,height=2.5in]{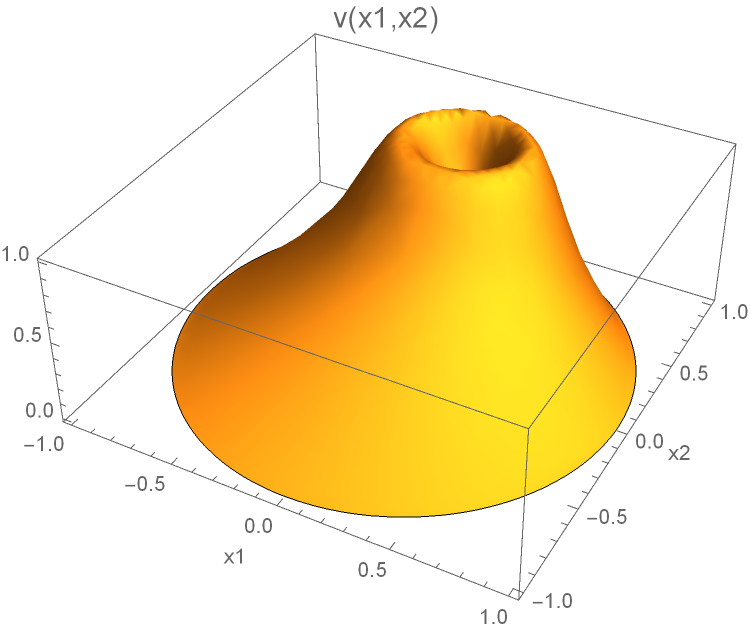} 
\caption{An illustration of the computed function $v(x_1,x_2)$ for Example \ref{ex5}.}
\label{fig-five_one}
\end{figure}

\begin{figure}
\center
\includegraphics[width=3.3in,height=1.8in]{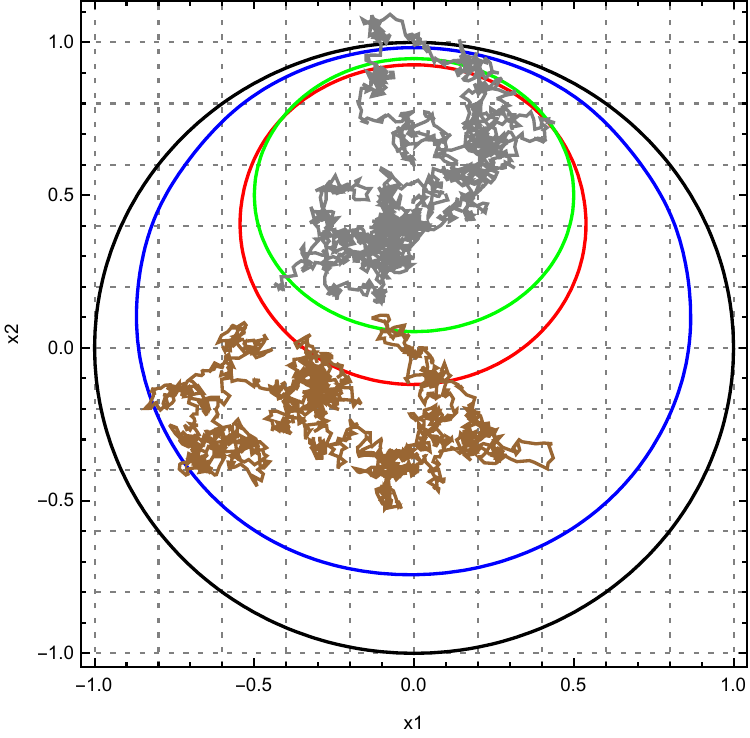} 
\caption{An illustration of inner-approximations of the reach-avoid set ${\rm RA}_p$ for Example \ref{ex5}. (Black and green curves denote the boundaries of the safe set $\mathcal{X}$ and the target set $\mathcal{T}$, respectively. Red and blue curves denote the boundaries of the computed 0.5- and  0.1-reach-avoid sets, respectively.)}
\label{fig-five1}
\end{figure}

\begin{figure}[htbp]
\centering
\includegraphics[width=1.5in,height=1.8in]{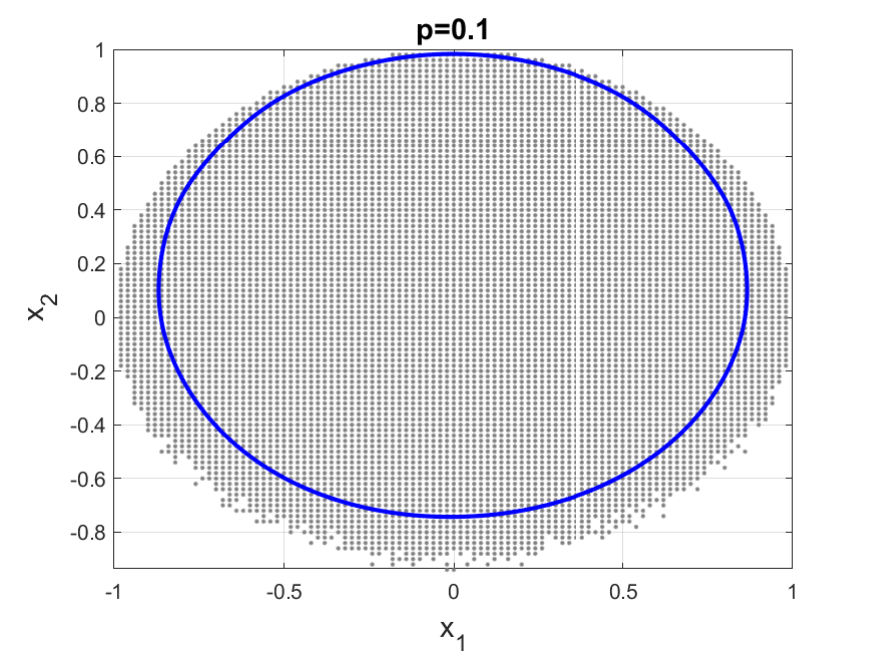} 
\includegraphics[width=1.5in,height=1.8in]{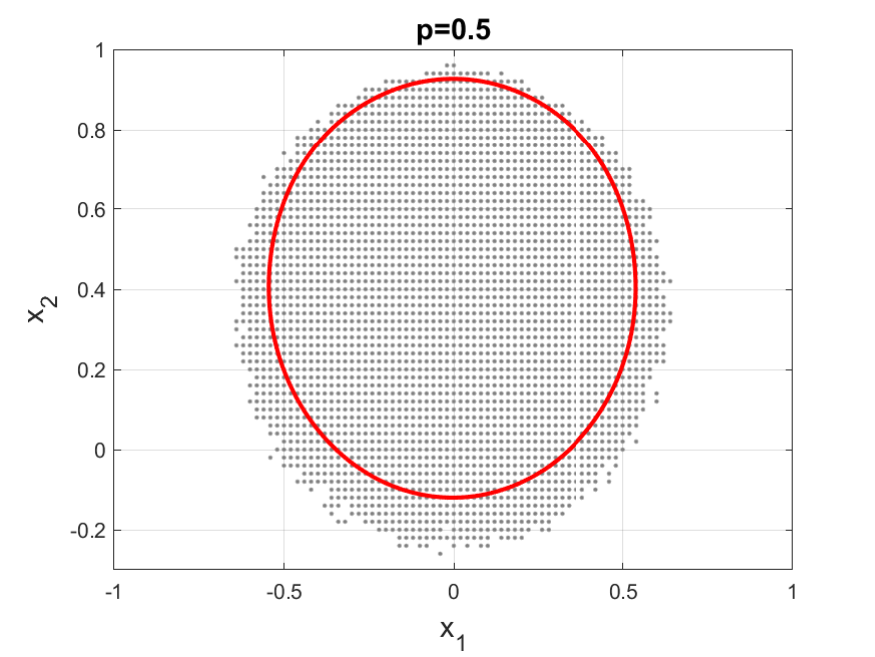} 
\caption{An illustration of the quality of computed inner-approximations of the reach-avoid set ${\rm RA}_p$ for Example \ref{ex5}. (Red and blue curves denote the boundaries of computed inner-approximations of the 0.5- and 0.1-reach-avoid sets, respectively. Gray region denotes the $0.5/0.1$-reach-avoid set estimated via the Monte-Carlo simulation method.)}
\label{fig_five2}
\end{figure}

\end{example}

\begin{example}
\label{ship}
As a model for the horizontal slow drift motions of a moored floating platform or ship responding to incoming irregular waves John Grue introduced the equation, 
\[\ddot{x}_t+a_0\dot{x}_t+\omega^2 x_t=(T_0-\alpha_0\dot{x}_t)\eta W_t \]
where $W_t$ is 1-dimensional white noise, $a_0,w,T_0,\alpha_0$ and $\eta$ are constants, which can be reduced to the following SDE
\begin{equation*}
\begin{split}
&d X_{1}(t,w)=X_{2}(t,w)dt\\
&d X_{2}(t,w)=(-\omega^2 X_{1}(t,w)-a_0 X_{2}(t,w)) dt\\
&~~~~~~~~~~~~~~~~~~~~~~~~~~~+(-\alpha_0 X_2(t,w)+T_0) \eta dW(t,w),
\end{split}
\end{equation*} 
where $\omega=a_0=\alpha_0=\eta=T_0=1.$

Suppose that the safe set and the target set are $\mathcal{X}=\{(x_1,x_2)^{\top}\in \mathbb{R}^2\mid x_1^2+x_2^2-1<0\}$ and $\mathcal{T}=\{(x_1,x_2)^{\top}\in \mathbb{R}^2\mid 4(x_1-0.2)^2+4(x_2-0.2)^2\leq 1\}$, respectively. 

The computed function $v(x_1,x_2)$ via solving the semi-definite program \eqref{sos} is shown in Fig. \ref{fig-ship_one} and the computed 0.1- and  0.5-reach-avoid sets are illustrated in Fig. \ref{fig-ship}. Three trajectories starting from $(-0.5,0.0)^{\top}$, $(-0.1,-0.5)^{\top}$ and $(0.1,0.9)^{\top}$ respectively are also illustrated in Fig. \ref{fig-ship}. Also, we use the Monte-Carlo simulation method to assess the conservativeness of computed inner-approximations, which is demonstrated in Fig. \ref{fig-ship1}.

\begin{figure}
\center
\includegraphics[width=3.3in,height=2.2in]{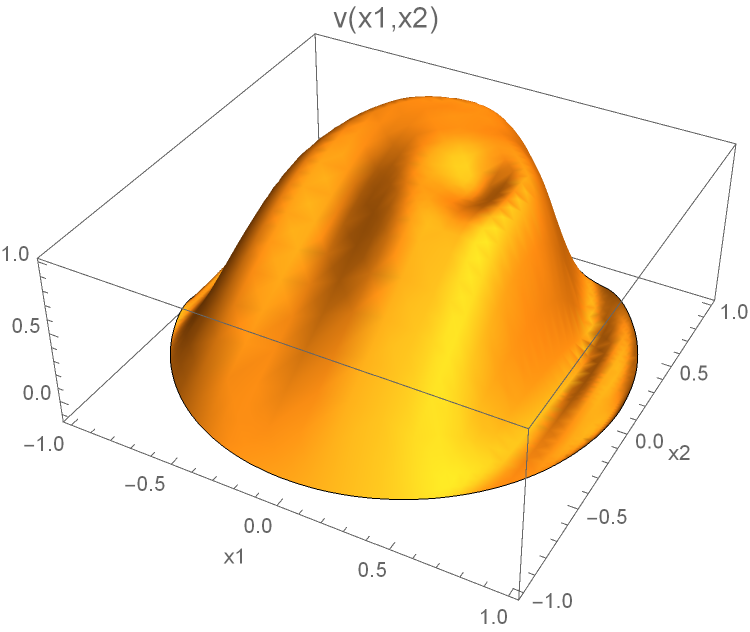} 
\caption{An illustration of the computed function $v(x_1,x_2)$ for Example \ref{ship}.}
\label{fig-ship_one}
\end{figure}
\begin{figure}
\center
\includegraphics[width=3.3in,height=1.8in]{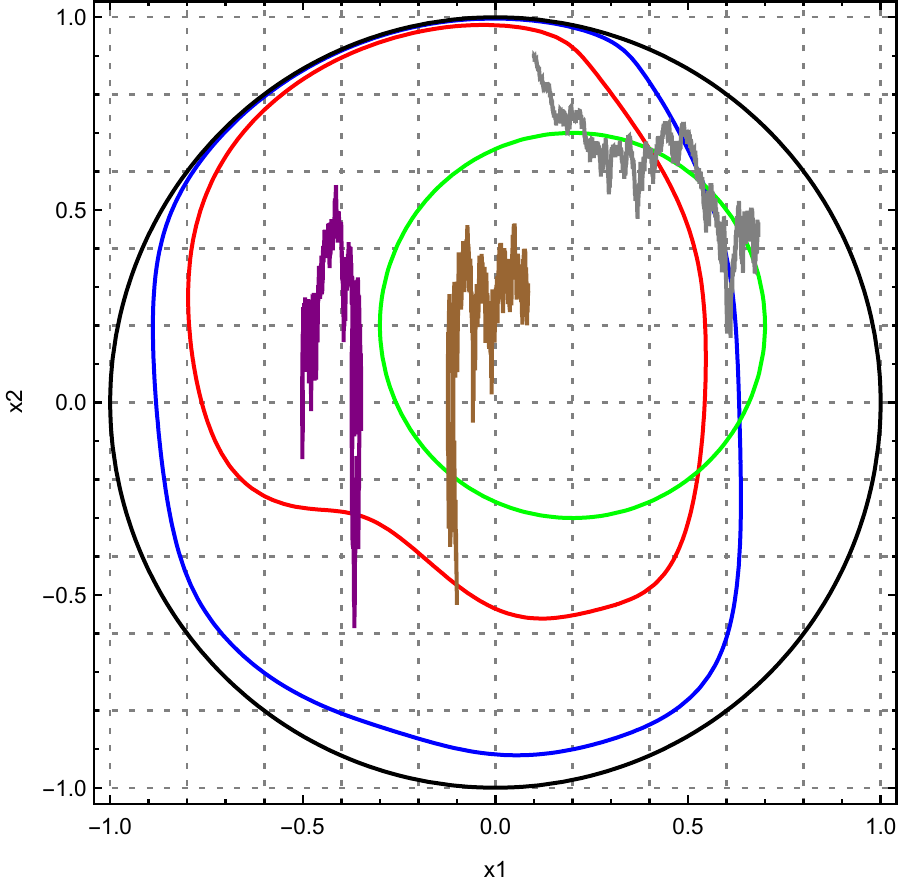} 
\caption{An illustration of inner-approximations of the reach-avoid set ${\rm RA}_p$ for Example \ref{ship}. (Black and green curves denote the boundaries of the safe set $\mathcal{X}$ and the target set $\mathcal{T}$, respectively. Red and blue curves denote the boundaries of the computed 0.5- and  0.1-reach-avoid sets, respectively.)}
\label{fig-ship}
\end{figure}

\begin{figure}[htbp]
\center
\includegraphics[width=1.5in,height=1.8in]{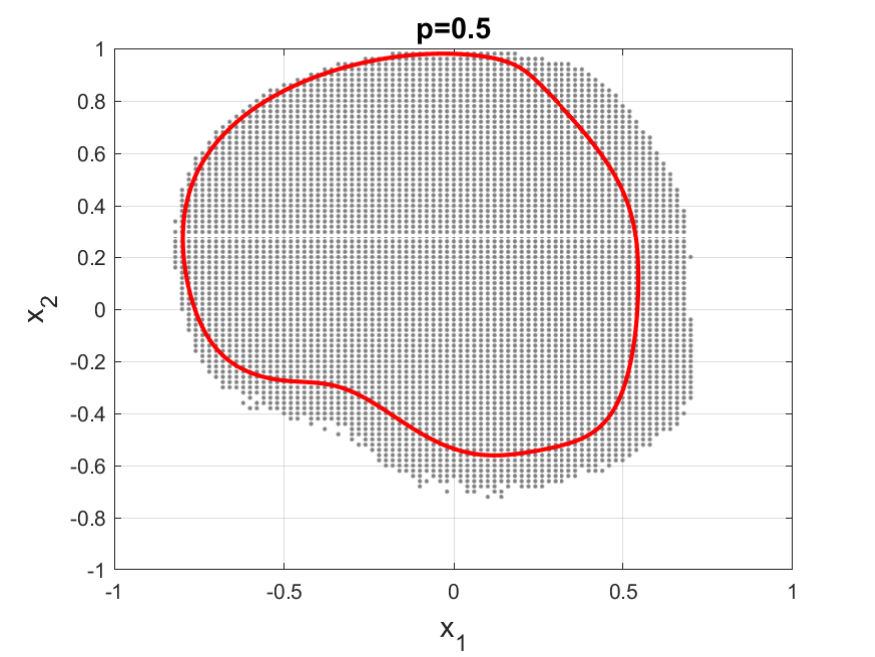} 
\includegraphics[width=1.5in,height=1.8in]{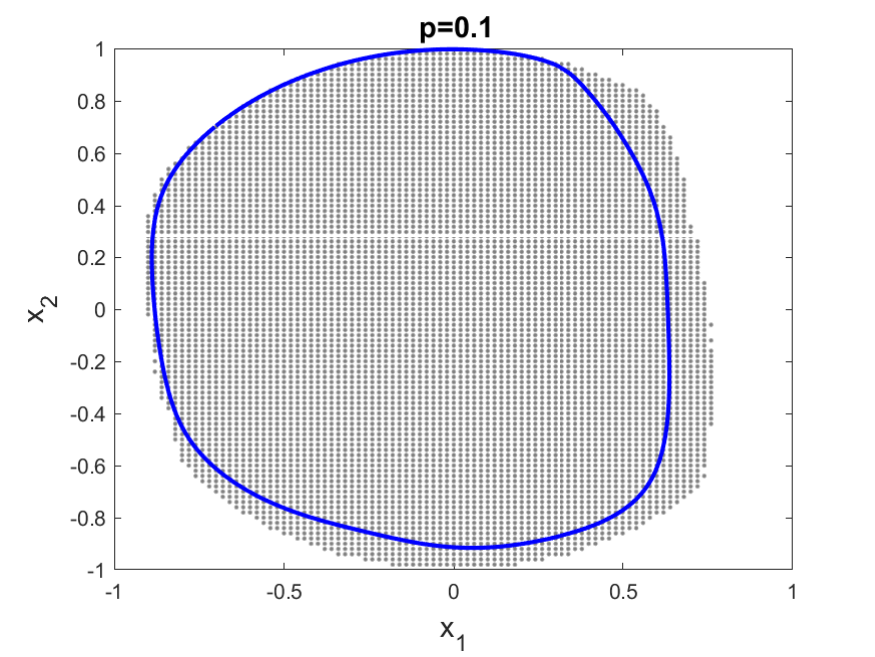}
\caption{An illustration of the quality of computed inner-approximations of the reach-avoid set ${\rm RA}_p$ for Example \ref{ship}. (Red and blue curves denote the boundaries of computed inner-approximations of the 0.5- and 0.1-reach-avoid sets, respectively. Gray region denotes the $0.9/0.5$-reach-avoid set estimated via the Monte-Carlo simulation method.)}
\label{fig-ship1}
\end{figure}
\end{example}

The proposed semi-definite programming method reduces the challenging (non-convex) problem of inner-approximating reach-avoid sets for polynomial SDEs into a convex optimization problem, which could be solved efficiently via interior point methods in polynomial time. In practice, computational cost can become prohibitive as either the dimension of SDEs or the polynomial degree of $v(\bm{x})$ and/or $u(\bm{x})$ increases, at least with the standard approach to the sum-of-squares optimization wherein generic semi-definite programs are solved by second-order symmetric interior-point algorithms. Large problems may be tackled using specialized nonsymmetric interior-point \cite{papp2019sum} or first order algorithms \cite{zheng2018fast}.

\section{Conclusion}
\label{sec:conclusion}
We have exposed and proved a correct algorithm based on semi-definite programming facilitating inner-approximations of $p$-reach-avoid sets of systems modeled by polynomial SDEs over open time horizons. As the $p$-reach-avoid set is the set of initial states forcing the system, with sufficient probability being larger than $p$, to eventually reach a desired target set while satisfying certain legal state constraints till the first hit time, it is of immediate interest in the design of reliable systems. The benchmark examples exposed in the previous section give an idea of the design or analysis obligations that can be answered by computation of safe, i.e., inner approximations of $p$-reach-avoid sets. They also demonstrate the performance of the proposed approach.

We would like to extend our method to solving reach-avoid problems of controlled SDEs and impulsive stochastic delay differential systems \cite{hu2019some}. 
\bibliographystyle{abbrv}
\bibliography{reference}

\end{document}